\documentclass[11pt,a4paper,oneside]{article}
\usepackage[utf8]{inputenc} \usepackage[T1]{fontenc}
\usepackage{lmodern}

\usepackage[english]{babel} 
\usepackage{geometry}
\usepackage{caption}
\usepackage{subcaption}
\usepackage{mathrsfs,amsmath,amsfonts,amstext,amssymb,amsthm,amsopn,amsthm,color, dsfont,bm,bbm}
\usepackage{enumitem}
\usepackage[titletoc]{appendix}

\usepackage{thmtools}
\usepackage[colorlinks=true, linkcolor=blue, citecolor=black]{hyperref}
\addto\extrasenglish{}
\addto\extrasenglish{}

\numberwithin{equation}{section}
\declaretheorem[name=Theorem, numberwithin=section]{thm}
\newtheorem{lem}[thm]{Lemma}
\newtheorem{prop}[thm]{Proposition}
\newtheorem{cor}[thm]{Corollary}

\newtheorem{defn}[thm]{Definition}

\declaretheoremstyle[bodyfont=\normalfont]{remark-style}
\declaretheorem[name={Remark}, style=remark-style, sibling=thm]{rem}
\declaretheorem[name={Example}, style=remark-style, sibling=thm]{exmp}

\newcommand{\R}{\mathbb{R}}
 
\newcommand{\Rd}{{\R^{d}}}
\newcommand{\N}{\mathbb{N}}

\newcommand{\dsubset}{\subset \subset}
\newcommand{\E}{\mathbb{E}}
\renewcommand{\P}{\mathbb{P}}
\renewcommand{\H}{\mathbb{H}}
\renewcommand{\AA}{\mathcal{A}}

\newcommand{\DD}{\mathcal{D}}

\newcommand{\LL}{\mathscr{L}}
\newcommand{\Lspace}{\mathcal{L}^1}
\newcommand{\KK}{\mathcal{K}}
\newcommand{\UU}{\mathcal{U}}

\renewcommand{\H}{\mathbb{H}} 
\newcommand{\ind}{{\bf 1}}
\renewcommand{\leq}{\leqslant} 

\renewcommand{\geq}{\geqslant}

\newcommand{\norm}[1]{\|#1\|}

\newcommand{\Ll}{L_{\operatorname{loc}}}
\newcommand{\Cl}{C_{\operatorname{loc}}}
\newcommand{\seto}{\{0\}}

 \DeclareMathOperator{\dist}{dist}
\DeclareMathOperator{\supp}{supp}
\DeclareMathOperator{\diam}{diam}

\def\({\left(} 
\def\){\right)} 
\def\[{\left[}
\def\]{\right]} 
\def\<{\langle} 
\def\>{\rangle}

\def\lv{\left\lvert}
\def\rv{\right\rvert}
\newcommand{\tauD}{\tau_D}

\reversemarginpar
\definecolor{ll}{rgb}{0.7,0.1,0.2}

\usepackage[backgroundcolor=white, bordercolor=red, linecolor=red]{todonotes}
\renewcommand{\d}{\, \textnormal{d}}

\AtBeginDocument{%
 \setlength\abovedisplayskip{2pt}
}

\begin{document}

\title{Remarks on the nonlocal Dirichlet problem\thanks{Research supported in 
part by National Science Centre (Poland) grant 2014/14/M/ST1/00600 and by the 
DFG through the CRCs 701 and 1283}} 
\author{Grzywny, Tomasz \footnote{Wrocław University of Science and 
		Technology, Faculty of Pure and Applied 
		Mathematics, 27 Wybrzeże 
		Wyspiańskiego
		50-370 Wrocław, Poland, \emph{email:}
		tomasz.grzywny@pwr.edu.pl}
	\and 
	Kassmann, Moritz \footnote{Universit\"{a}t Bielefeld, Fakult\"{a}t f\"{u}r 
		Mathematik, Postfach 100131, D-33501 Bielefeld, Germany, \emph{email:} 
		moritz.kassmann@uni-bielefeld.de}
	\and 
	Le\.{z}aj, \L{}ukasz \footnote{Wrocław University of Science and 
		Technology, Faculty of Pure and Applied 
		Mathematics, 27 Wybrzeże 
		Wyspiańskiego
		50-370 Wrocław, Poland, 
		\emph{email:} lukasz.lezaj@pwr.edu.pl} 
}

\maketitle

\begin{abstract}
We study translation-invariant integrodifferential operators that generate 
L\'{e}vy processes. First, we investigate different notions of what a solution 
to a nonlocal Dirichlet problem is and we provide the classical 
representation formula for distributional solutions. Second, we study the 
question under which assumptions distributional solutions are twice 
differentiable in the classical sense. Sufficient conditions and 
counterexamples are provided.  
\end{abstract}


MSC2010 Subject Classification: 34B05, 47G20, 60J45

\section{Introduction}

The aim of this article is to provide two results on translation-invariant 
integrodifferential operators, which are not surprising but have not been 
systematically covered in the literature. Let us briefly explain these results 
in case of the classical Laplace operator.

\medskip

The classical result of Weyl says the following. Assume $D 
\subset \R^d$ is an open set, $f \in C^\infty(D)$, and $u \in \mathcal{D}'(D)$ 
is a Schwartz distribution satisfying $\Delta u = f$ in the distributional 
sense, i.e. $\langle u, \Delta \psi \rangle = \langle \psi, f 
\rangle$ for every $\psi \in C^\infty_c(D)$. Then $u \in C^\infty (D)$ and 
$\Delta u = f$  in $D$. This is the starting point for 
the study of distributional solutions to boundary value problems. Our first aim 
is to study distributional solutions to nonlocal boundary value problems of the 
form
\begin{alignat*}{2}
\LL u &= f \quad &&\text{in } D\,, \\
u &= g &&\text{in } D^c\,,
\end{alignat*} 
where $\LL$ is an integrodifferential operator generating a unimodal L\'evy 
process. Our second aim is to provide sufficient conditions such that 
distributional solutions $u$ to the nonlocal Dirichlet problem are twice 
differentiable in the classical sense. In case of the Laplace operator, it is 
well known that Dini continuity of $f: D \to \R$, i.e. finiteness of the 
integral $\int_0^1 \omega_f(r)/r \d r$ for the modulus of continuity 
$\omega_f$, implies that the distributional solution $u$ to the classical 
Dirichlet problem satisfies $u \in C^2_{\operatorname{loc}}(D)$. On 
the other hand, one can construct a continuous function $f:B_1 \to \R$ and 
a distribution $u \in \mathcal{D}'(B_1)$ such that $\Delta u = f$ in the 
distributional sense, but $u \notin C^2_{\operatorname{loc}}(B_1)$. These 
observation have been made long time ago \cite{HaWi55}. They have been 
extended to non-translation-invariant operators by several 
authors \cite{MR521856, MaEi71} and to nonlinear problems \cite{Kov97, 
	DGM04}. Note that there are many more related contributions including 
treatments 
of partial differential equations on non-smooth domains. In the present work we 
treat the simple linear case for a general class of nonlocal operators 
generating unimodal L\'evy processes.

\medskip

Let us introduce the objects of our study and formulate our main 
results. Let $\nu\!: \R^d\setminus\{0\} \to [0,\infty)$ be a function satisfying
\begin{align*} 
\int \big( 1 \wedge |h|^2\big) \nu(h) \d h < \infty \,. 
\end{align*} 
The function $\nu$ induces a measure $\nu(\! \d h) = \nu(h) \d 
h$, which is called the L\'{e}vy measure. Note that we use the same symbol for the 
measure as well as for the density. We study operators of the form
\begin{align}\label{eq:def-L}
\LL u(x) = \lim_{\epsilon \to 0} \int_{|y|>\epsilon} (u(x+y)-u(x))\nu(y) \d 
y\,. 
\end{align}
This expression is well defined if $u$ is sufficiently regular in the 
neighbourhood of $x \in \R^d$ and satisfies some integrability condition at 
infinity. We recall that for $\alpha \in (0,2)$ and $\nu(\!\d 
h)= c_\alpha |h|^{-d-\alpha} \d h$ with some appropriate constant $c_\alpha$, 
the operator $\LL$ equals the fractional Laplace operator 
$-(-\Delta)^{\alpha/2}$ on $C_b^2(\R^d)$. The regularity theory of such 
operators has been intensively studied recently. For instance, it is well known 
\cite{MR2555009, MR3168912, MR3293447, MR3536990, MR3447732} that the solution 
of 
$-(-\Delta)^{\alpha/2} 
u=f$ with $f \in C^{\beta}$ belongs to $C^{\alpha+\beta}$ provided that 
neither  $\beta$ nor $\alpha+\beta$ is an integer. The same result in more 
general setting is derived in \cite{BK2015}.

\medskip

Our standing assumption is that $h \to \nu(h)$ is a 
non-increasing radial function and that there exists a L\'{e}vy measure 
$\nu^*$ resp. a density $\nu^*$ such that $\nu \leq \nu^*$ and 
\begin{align}\label{growth_condition}
\nu^*(r) \leq C\nu^*(r+1), \quad r\geq r_0
\end{align}
for some $r_0, C \geq 1$. Given an open set $D \subset \R^d$, 
denote by $\Lspace(D)$ the vector space of all Borel functions $u \in \Ll^1$ 
satisfying
\begin{align}\label{measure_scaling}
\int_D |u(x)| (1 \wedge \nu^*(x)) \d x < \infty.
\end{align}
The condition $u \in 
\Lspace(D)$ is the integrability condition needed to ensure well-posedness in 
the definition of $\LL u$ in distributional sense. Given an open set, we 
denote by $G_D$ resp. $P_D$ the usual Green resp. the Poisson operator, cf. 
\autoref{sec:prelims}. For a definition of the Kato class $\KK$ and $\KK(D)$ see 
\autoref{def:Kato_class} below. Here is our first result.

\begin{thm}\label{thm:weak_thm}
	Let $D$ be a bounded open set. Suppose $f \in L^1(D)$ and $g \in \Lspace(D^c)$. 
	Let $u \in \Lspace(\Rd) $ be a distributional solution of the Dirichlet 
problem
	
	\begin{align}\label{eq:weak_problem}
	\begin{array}{rlll}
	\LL u &=& f & \text{in } D\,, \\
	u &=& g & \text{in } D^c\,. 
	\end{array} 
	\end{align} 
	
	Then $u(x) + G_D[f](x)$ satisfies the mean-value property inside 
	$D$. Furthermore, if $D$ is a Lipschitz domain and there exists $V 
\dsubset D$ such that $f$ and $g \ast \nu$ belongs to the Kato class $\KK(D 
\setminus \overline{V})$, then there is a unique solution which is bounded 
close to the boundary of $D$
	\begin{align*}
	u(x) =  - G_D[f](x)+P_D[g](x).
	\end{align*}
\end{thm}
The theorem above says that the 
distributional solution of \eqref{eq:weak_problem} is unique up to a harmonic 
function. If, additionally,  $D$ is a Lipschitz domain and we impose some 
regularity, then the solution is unique. Boundedness of $u$, $f$, $g$ would 
suffice, of course. It is obvious that one has to impose some regularity 
condition on $f$ in order 
to prove uniqueness of solutions. Note that, in the case where $\LL$ equals the 
fractional Laplace operator, similar results like \autoref{thm:weak_thm} are 
proved in \cite{MR1671973}. A result similar to \autoref{thm:weak_thm} has 
recently been proved in \cite{KKLL2018}. The authors consider a smaller class 
of operators and concentrate on viscosity solutions instead of distributional 
solutions. 

\medskip

Variational solutions to nonlocal operators have been studied by several 
authors, e.g., in \cite{MR3318251, MR3738190}. The problem to determine 
appropriate function spaces for the data $g$ leads to the notion of nonlocal 
traces spaces introduced in \cite{DyKa16}. It is interesting that the study of 
Dirichlet problems for nonlocal operators leads to new questions regarding the 
theory of function spaces. 

\medskip

The formulation of our second main result requires some further preparation. 
They are rather technical because we cover a large class of 
translation-invariant operators. The similar condition to the following appears in  
\cite{BGPR2017}. 
\begin{enumerate}
	\item[(A)] $\nu$ is twice continuously differentiable and 
	there is a positive constant $C$ such that
	\begin{align*}
	|\nu'(r)|, |\nu''(r)| \leq C \nu^*(r) \quad \text{ for } r \geq r_0.
	\end{align*}
\end{enumerate}
(A) and \eqref{growth_condition} are essential for proving that functions with 
the mean-value property are twice continuously differentiable, see 
\autoref{lem:harm_c2}. We emphasize that in general this is not the case and 
usually harmonic functions lack sufficient regularity if no additional 
assumptions are imposed. The reader is referred to \cite[Example 
$7.5$]{MR3413864}, where a function $f$ with the mean-value property is 
constructed for which $f'(0)$ does 
not exist.   

\medskip

Let $G$ be a fundamental solution of $\LL$ on $\Rd$ (see \eqref{G_def} for 
definition). Note that in the case of the fractional Laplace operator $G(x) 
=c_{d,\alpha} |x|^{\alpha - d}$ for $d\neq \alpha$ and some constant 
$c_{d,\alpha}$. In what follows we will assume the kernel $G$ to satisfy the 
following growth condition:
\begin{enumerate}
	\item[(G)] $G \in C^2(\Rd \setminus \seto)$ and there exists a 
non-increasing function $S\!: \, (0,\infty) \mapsto [0,\infty)$ and $r_0>0$ 
such 
that
	\begin{enumerate}
		\item[(i)] if $\int_0^{1/2}|G'(t)|t^{d-1} \d t=\infty$, then
		\begin{align*}
		G(r),|G'(r)|,r|G''(r)| \leq S(r), \quad r<r_0,
		\end{align*}
		\item[(ii)] if $\int_0^{1/2}|G'(t)|t^{d-1} \d t<\infty$, then 
additionally $G \in C^3(\Rd \setminus \seto)$ and
		\begin{align*}
		G(r),|G'(r)|,|G''(r)|,r|G'''(r)| \leq S(r), \quad r<r_0.
		\end{align*}
	\end{enumerate}
\end{enumerate}

\begin{thm}\label{thm:main_thm}
	Let $D$ be an open bounded set. Assume that 
the measure $\nu$ satisfies (A) and \eqref{growth_condition} and the 
fundamental solution $G$ satisfies (G). Let $g \in \Lspace(D^c)$ and $f\!: \, D 
\mapsto \R$.  If $\int_0^1 |G'(t)|t^{d-1}\d t < \infty$ we assume 
\begin{align}\label{main_thm_cond1}
\int_0^{1/2} S(t) \omega_f(t,D) t^{d-1}\d t < 
\infty,
\end{align}  
or  if $\int_0^1 |G'(t)|t^{d-1}\d t = \infty$ we assume
\begin{align}\label{main_thm_cond2}
\int_0^{1/2} S(t) \omega_{\nabla f}(t,D) t^{d-1}\d t < \infty\,.
\end{align}

Then the solution $u \in \Lspace(\Rd)$ of the problem
	\begin{align}\label{General_problem3}
	\left\{ \begin{array}{rlll}
	\LL u &=& f & \text{in } D, \\
	u &=& g & \text{in } D^c. 
	\end{array} \right.
	\end{align}
	belongs to $\Cl^2(D)$ and is unique up to a harmonic function (with 
respect to 
	$\LL$).
\end{thm}

\begin{rem}
	\eqref{main_thm_cond1} or \eqref{main_thm_cond2} imply $f \in \KK(D)$, 
so by \autoref{thm:weak_thm}, if $D$ is a Lipschitz domain and $g \ast v \in 
\KK(D)$ then the solution is unique. 
\end{rem}

The result uses quite involved conditions because the 
measure $\nu$ interacts with the Dini-type assumptions for the right-hand side 
function $f$. Looking at examples, we see that the two cases 
described in the theorem appear naturally. In the fractional Laplacian case 
($G(x) =c_{d,\alpha} |x|^{\alpha - d}$), finiteness of the expression 
$\int_0^{1/2} |G'(t)|t^{d-1}\d t$ depends on the value of $\alpha \in (0,2)$. 
We show in \autoref{sec:examples} that the conditions hold true 
when $\LL$ is the generator of a rotationally symmetric $\alpha$-stable 
process, i.e., when $\LL$ equals the fractional Laplace operator. Note that 
\autoref{thm:main_thm} is a new result even in this case. We also study the 
more general class, e.g. operators of the form $-\varphi(-\Delta)$, where 
$\varphi$ is a Bernstein function. Note that in the theorem above we do not 
assume that $g$ is bounded.

\medskip

\begin{rem}
We emphasize that in the case of $\LL$ being the fractional Laplace operator of 
order $\alpha \in (0,2)$ and $f \in C_{\operatorname{loc}}^{2-\alpha}(D)$, 
it is not true that every solution of $\LL u = f$ belongs to
$C_{\operatorname{loc}}^2(D)$ as is stated in \cite[Theorem $3.7$]{AJS2018}. A 
similar phenomenon has been mentioned in \cite{MR2555009} and is visible 
here as well. Observe that in such case the integrals \eqref{main_thm_cond1} 
and \eqref{main_thm_cond2} are clearly divergent and consequently, 
\autoref{thm:main_thm} cannot be applied. We devote 
\autoref{sec:counterexamples} to the construction of counterexamples for any 
$\alpha \in (0,2)$. 
\end{rem}

\medskip

The article is organized as follows: in \autoref{sec:prelims} we provide the 
main definitions and some preliminary results. The proof of 
\autoref{thm:weak_thm} is provided in \autoref{sec:weak-solutions}.  
\autoref{sec:sufficient_condition} contains several rather 
technical computations and the proof of \autoref{thm:main_thm}. We discuss the 
necessity of the assumptions of \autoref{thm:main_thm} through examples in 
\autoref{sec:counterexamples}. Finally, in \autoref{sec:examples} we 
provide examples that show that the assumptions of \autoref{thm:main_thm} are 
natural.
\section{Preliminaries}\label{sec:prelims}

In this section we explain our use of notation, define several objects and 
collect some basic facts. We write $f \asymp g$ when $f$ and $g$ are 
comparable, that is the quotient $f/g$ stays between two positive constants. To 
simplify the notation, for a radial function $f$ we use the same symbol to 
denote its radial profile. In the whole paper $c$ and $C$ denote constants 
which may vary from line to line. We write $c(a)$ when the constant $c$ depends 
only on $a$. By $B(x,r)$ we denote the ball of radius $r$ centered at $x$, that 
is $B(x,r)=\{y \in \Rd: \ |y-x|<r\}$. For convenience we set $B_r=B(0,r)$. For 
an open set $D$ and $x \in D$ we define $\delta_D(x)=\dist(x,\partial D)$ and 
$\diam(D)=\sup_{x,y \in D} 
|x-y|$. The modulus of continuity of a continuous function $f: D \to \R$ 
is defined by
\begin{align*}
\omega_{f}(t,D) = \sup \{|f(x) - f(y)| : \; x,y \in D, |x-y| < t\} \quad (t > 
0)\,.
\end{align*}

For a differentiable function $f: D \to \R$ we set
\begin{align*}
\omega_{\nabla f}(t,D) = \max\limits_{i \in \{1, \ldots, d\}} \sup 
\{|\partial_{x_i} f(x) - \partial_{x_i}  f(y)| : \; x,y \in D, |x-y| < t\} 
\quad (t > 
0)\,.
\end{align*}

\medskip

We say that a Borel measure is isotropic unimodal if it is absolutely 
continuous 
on $\Rd \setminus \seto$ with respect to the Lebesgue measure and has a radial, 
non-increasing density. Given an isotropic unimodal L\'{e}vy measure $\nu(\! \d 
x)=\nu(|x|) \d x$, we define a L\'{e}vy-Khinchine exponent
\begin{align*}
\psi(\xi) = \int_{\Rd} \(1-\cos(\xi \cdot x)\) \nu(\! \d x), \quad \xi \in \Rd. 
\end{align*}

$\psi$ is usually called \emph{the characteristic exponent}. It is well known 
(e.g. \cite[Lemma $2.5$]{MR3413864}) 
that if $\nu(\Rd)=\infty$, there exist a continuous function $p_t \geq 0$ in 
$\Rd \setminus \seto$ such that 
\begin{align*}
\widehat{p_t}(\xi)=\int_{\Rd} e^{-i\xi \cdot x} p_t(x) \d x=e^{-t \psi(\xi)}, 
\quad \xi \in \Rd. 
\end{align*}
The family $\{p_t\}_{t > 0}$ induces a strongly continuous contraction semigroup 
on 
$C_0(\Rd)$ and $L^2(\Rd)$
\begin{align*}
P_t f(x) = \int_{\Rd} f(y)p_t(y-x) \d y, \quad x \in \Rd,
\end{align*}
whose generator $\AA$ has the Fourier symbol $-\psi$. Using the Kolmogorov 
theorem one can construct a stochastic process $X_t$ with transition densities 
$p_t(x,y)=p_t(y-x)$, namely $\P^x(X_t \in A)=\int_A p_t(x,y) \d y$. Here $\P^x$ 
is the probability corresponding to a process $X_t$ starting from $x$, that is 
$\P^x(X_0=x)=1$. By $\E^x$ we denote the corresponding expectation. In fact, 
$X_t$ is a pure-jump isotropic unimodal L\'{e}vy process in $\Rd$, that is a 
stochastic process with stationary and independent increments and 
c\`{a}dl\`{a}g paths (see for instance \cite{MR1739520}). 

\medskip
 
One of the objects of significant importance in this paper is the potential 
kernel defined as follows:
\begin{align*}
U(x,y) = \int_0^{\infty} p_t(x,y) \d y.
\end{align*}
Clearly $U(x,y)=U(y-x)$. The potential kernel can be defined in our setting if 
$\int_{B_1}\frac{1}{\psi(\xi)}d\xi<\infty$. In particular, for $d \geq 3$ the 
potential kernel always exists (see \cite[Theorem 37.8]{MR1739520}). If this is 
not the case, one can consider the compensated potential kernel 
\begin{align}\label{def-compensated_kernel}
W_{x_0}(x-y)=\int_0^{\infty}\(p_t(x-y) - p_t(x_0)\) \d t
\end{align}
for some fixed $x_0 \in \Rd$. If $d=1$ and $\int_{B_1} \frac{\d 
\xi}{\psi(\xi)}<\infty$, we can set $x_0=0$. In other cases the compensation 
must be taken with $x_0 \in \Rd \setminus \seto$. For details we refer the 
reader to \cite{MR3636597} and to the \autoref{app:appendix}.

\medskip

Slightly abusing the notation, we let $W_1$ be \eqref{def-compensated_kernel} 
for $x_0=(0,...,0,1) \in \Rd$. Thus, we have arrived with three potential 
kernels: $U$, $W_0$ and $W_1$. Each one corresponds to a different type of 
process $X_t$ and an operator associated with it. In order to merge these cases 
in one object, we let
\begin{align}\label{G_def}
G(x) = \left\lbrace
\begin{array}{rl}
U(x), & \text{if } \int_{B_1} \frac{\d \xi}{\psi(\xi)}<\infty, \\
W_0(x), & \text{if } $d=1$, \int_{B_1} \frac{\d \xi}{\psi(\xi)}=\infty \text{ 
and } \int_0^{\infty} \frac{1}{1+\psi(\xi)} \d \xi< \infty, \\
W_1(x), & \text{otherwise.}
\end{array} \right.
\end{align}
For instance, in the case of $\LL=\Delta$ we have
\begin{align*}
G(x)=\left\{
\begin{array}{ll}
c_d |x|^{2-d}, & d \geq 3,\\
\frac{1}{\pi} \ln \frac{1}{|x|}, & d=2,\\
|x|, & d=1.
\end{array} \right.
\end{align*}

\medskip

The basic object in the theory of stochastic processes is the first exit time 
of $X$ from $D$,
\begin{align*}
\tauD=\inf \{ t>0\!: \, X_t \notin D\}.
\end{align*}
Using $\tauD$ we define an analogue of the generator of $X_t$, namely, the 
\emph{characteristic operator} or \emph{Dynkin operator}. We say a Borel 
function $f$ is in a domain $\DD_{\UU}$ of Dynkin operator $\UU$ if there 
exists
a limit  
\begin{align*}
\UU f(x) = \lim_{B \to \{x\}} \frac{\E^x \( X_{\tau_B} \)-f(x)}{\E^x \tau_B}.
\end{align*}
Here $B \to \{x\}$ is understood as a limit over all sequences of open sets 
$B_n$ whose intersection is $\{x\}$ and whose diameters tend to $0$ as $n \to 
\infty$. The characteristic operator is an extension of $\AA$, that is 
$\DD_{\AA} \subset \DD_{\UU}$ and $\UU\vert_{\DD_{\AA}}=\AA$. For a wide 
description of characteristic operator and its relation with the generator of 
$X_t$ we refer the reader to \cite[Chapter V]{MR0193671}. 

\medskip
 
Instead of the whole $\Rd$, one can consider a process $X$ killed after exiting 
$D$. By $p_t^D(x,y)$ we denote its transition density (or, in other words,  the 
fundamental solution of $\partial_t-\LL$ in $D$). We have 
\begin{align*}
p_t^D(x,y)=p_t(x,y)-\E^x[\tauD<t;p_{t-\tauD}(X_{\tauD},y)], \quad x,y\in \Rd.
\end{align*}
It follows that $0\leq p_t^D \leq p_t.$ By $P_D(x,\d z)$ we denote the 
distribution of $X_{\tauD}$ with respect to 
$\P^x$, that is $P_D(x,A) = \P^x(X_{\tauD} \in A)$. We call $P_D(x,\d z)$ a 
\emph{harmonic measure} and its density $P_D(x,z)$ on $\Rd \setminus 
\overline{D}$ with respect to the Lebesgue measure --- a \emph{Poisson kernel}. 
For $g: D^c \mapsto \R$ we let
\begin{align*}
P_D[g](x) = \int_{D^c} g(z) P_D(x,\d z), \quad x \in D,
\end{align*}
if the integral exists. For $x \in D^c$ we set $P_D[g](x)=g(x)$. 

\begin{rem}\label{rem:PDg_Lspace}
	If $D$ is an open bounded set and $g \in \Lspace(D^c)$ then $P_D[g] \in 
\Lspace$. Indeed, since $P_D[g] \equiv g$ on $D^c$, it is enough to prove that 
$P_D[g] \in L^1(D)$. By the mean-value property, for any $B \dsubset D$ we have 
$P_B[P_D[g]](x)=P_D[g](x)$ for $x \in B$. It follows, by the Ikeda-Watanabe 
formula, that for any fixed $x \in B$
	\begin{align*}
 \infty>\int_{B^c} P_B(x,z) P_D[g](z) \d z \geq c \int_{A \cap D} P_D[g](z) \d 
z,
	\end{align*}
	where $A=B^c\cap(B+\mathrm{supp}(\nu)/2)$. 	Arbitrary choice of $B$ 
yields the claim.
\end{rem}

We define a Green function for the set $D$
\begin{align*}
G_D(x,y)=\int_0^{\infty} p_t^D(x,y) \d y, \quad x,y \in D,
\end{align*}
and the Green operator
\begin{align*}
G_D[f](x) = \int_D G_D(x,y)f(y)\d y.
\end{align*}
We note that $G_D(x,y)$ can be interpreted as the occupation time density up 
to the exit time $\tauD$, $G_D[f]$ --- as a mean value of $f(X_t)$. Using that 
we obtain $G_D[\ind]=\E^x \tauD$. For bounded sets $D$ we have $\sup_{x \in 
\Rd} \E^x \tauD < \infty$ (\cite{MR632968}, \cite{MR3350043}). By the strong 
Markov property for any open $\Omega\subset D$ we have
\begin{align}\label{G_D_subset}
G_D(x.y)=G_\Omega(x,y)+\E^xG_D(X_{\tau_\Omega},y), \quad x,y \in \Omega.
\end{align}
Obviously we 
have $G_{\Rd}=U$. If $U$ is well-defined (finite) a.s., the well-known Hunt 
formula holds: 
\begin{align*}
G_D(x,y)=U(y-x)-\E^x U(y-X_{\tauD}), \quad x,y \in D.
\end{align*}
In case of compensated potential kernels, a similar formula is valid, namely,
\begin{align}\label{eq:Hunt}
G_D(x,y)=G(y-x)-\E^x G(y-X_{\tauD}), \quad x,y \in D.
\end{align}
See \autoref{thm:recurrent_sweeping_formula}. 
\begin{defn}\label{mean_value_property}
	We say that a function $g:\R^d \to \R$ satisfies the mean-value 
	property in an open set $D \subset \R^d$ if $g(x)=P_D[g](x)$ for all 
	$x \in D$. Here we assume that the integral is absolutely convergent. 
If 
$g$ has the mean-value property in every bounded open set whose  
	closure is contained in $D$ then $u$ is said to have the mean-value 
property 
	inside $D$.
\end{defn}

Clearly if $f$ has the mean-value property inside $D$, then $\UU f=0$ in $D$.

\medskip

In general, functions with the mean-value property lack sufficient regularity 
if no additional assumptions are imposed. In our setting, however, we can show 
that they are, in fact, twice continuously differentiable in $D$.

\begin{lem}\label{lem:harm_c2}
	Let $g \in \Lspace$ and $D$ be an open set. Suppose that (A) and 
\eqref{growth_condition} hold. If $g$ has the mean-value property inside $D$, 
then $g \in C^2_{\operatorname{loc}}(D)$.
\end{lem}

The proof is similar to the proof of \cite[Theorem 4.6]{BGPR2017} and is 
omitted.

\begin{defn}[\cite{MR1132313}, \cite{MR3713578}] \label{def:Kato_class}
	We say that a Borel function $f$ belongs to the Kato class $\KK$ if it 
	satisfies the following condition
	\begin{align}\label{Kato_class}
	\lim_{r \to 0} \left[ \sup_{x \in \Rd} \int^r_0 P_t|f|(x) \d t \right] 
=0.
	\end{align}
	We say that $f \in \KK(D)$, where $D$ is an open set, if $f \ind_D \in 
	\KK$.
\end{defn}
This is one of three conditions discussed by Zhao in \cite{MR1132313}.
A detailed description of different notions of the Kato class and 
related conditions can be found in \cite{MR3713578}.

\begin{lem}\label{lem:Kato_green_op_bounded}
Let $V \dsubset D$ and $\rho:=\dist(V,\partial D)$. Suppose $f \in \KK(D 
\setminus \overline{V})$. Then $G_D[f]$ is bounded in $V_1:=\{x \in D \setminus 
V\! : 
\delta_D(x)<\rho/2\}$. 
\end{lem}
\begin{proof}
	Let $x \in V_1$ and define $V_2:=\{x \in D \setminus V\! : 
\delta_D(x)<3\rho/4\}$. We have 
	\begin{align*}
	\lv G_D[f\ind_{V_2^c}](x) \rv \leq \int_{V_2^c} G_D(x,y) \lv f(y) \rv 
\d 
y.
	\end{align*}
	Let $r=2\sup_{x \in D}|x|$. Then $D \subset B_r$ and by \cite[Theorem 
$1.3$]{MR3729529}
	\begin{align*}
	\int_{V_2^c} G_D(x,y) \lv f(y) \rv \d y \leq \int_{V_2^c} G_{B_r}(x,y) 
\lv f(y) \rv \d y \leq c(\rho) \norm{f}_1.
	\end{align*}
	Moreover, by \eqref{G_D_subset}
	\begin{align*}
	G_D[f \ind_{V_2}](x) = G_{D\setminus V}[f\ind_{V_2}](x)+\E^x 
G_D[f\ind_{V_2}]\( X_{\tau_{D \setminus V}} \).
	\end{align*}
	Observe that
	\begin{align*}
	\lv \E^x G_D [f \ind_{V_2}](X_{\tau_{D\setminus V}}) \rv \leq \E^x 
\int_{V_2} G_D(X_{\tau_{D\setminus V}},y) \lv f(y) \rv \d y \leq c(\rho/4) 
\norm{f}_1
	\end{align*}
	again by \cite[Theorem $1.3$]{MR3729529}. Finally, we have
	\begin{align*}
	\lv G_{D \setminus V}[f\ind_{V_2}](x) \rv \leq G_{D \setminus V}[\lv f 
\rv \ind_{D \setminus V} ](x).
	\end{align*}
	A straightforward application of the proof of \cite[Theorem 
$4.3$]{MR1329992} to the last term gives the claim. 
\end{proof}

\begin{prop}
	If $f$ satisfies \eqref{main_thm_cond1} then it is uniformly continuous 
in $D$. If \eqref{main_thm_cond2} holds then $\frac{\partial}{\partial x_i}f$, 
$i=1,...,d$, is uniformly continuous in $D$.
\end{prop}
\begin{proof}
	Suppose $\frac{\partial}{\partial x_i}f$ for some $i=1,...,d$ is not 
uniformly continuous, i.e. $\omega_{\nabla f}(t,D)\geq c>0$ for $t \leq 1$. If 
\eqref{main_thm_cond2} holds then in particular
	\begin{align*}
	\infty >  \int_0^{1/2} S(t) \omega_{\nabla f}(t,D) t^{d-1}\d t \geq c 
\int_0^{1/2} |G'(t)|t^{d-1} \d t,
	\end{align*} 
	which is a contradiction. Now let $\omega_f(t,D) \geq c$ for $t \leq 
1$, and suppose \eqref{main_thm_cond1}. For $d \geq 3$ we have
	\begin{align*}
	\infty > \int_0^{1/2} S(t) \omega_f(t,D) t^{d-1}\d t \geq c\int_0^{1/2} 
|G''(t)| t^{d-1}\d t.
	\end{align*}
	By integration by parts
	\begin{align*}
	\int_0^{1/2} G''(t) t^{d-1}\d t = 
G'(t)t^{d-1}\Big|_0^{1/2}-(d-1)\int_0^{1/2}G'(t)t^{d-2} \d t.
	\end{align*}
	Observe that $G'$ is of constant sign. Hence, both $\lim_{t \to 0^+} 
G'(t)t^{d-1}$ and the integral are  finite. In particular, integration by parts 
once again yields
	\begin{align*}
	\int_0^{1/2} G'(t)t^{d-2} \d t 
=G(t)t^{d-2}\Big|_0^{1/2}-(d-2)\int_0^{1/2} G(t)t^{d-3} \d t.
	\end{align*}
	Both $\lim_{t \to 0^+} G(t)t^{d-2}$ and the integral are positive. 
Hence, both must be finite. By \cite[Proposition $1$ and $2$]{MR3225805} we 
have $\int^r_{0}G(t)t^{d-1}dt \geq c \psi(1/r)^{-1}$. It follows that 
	\begin{align*}
	\int_0^1 G(t)t^{d-3}\,dt &= \int_{B_1} \frac{G(|x|)}{|x|^2} \d x = 
\int_{B_1} 
	\int_{|x|}^{\infty}\frac{1}{s^3}\d s \ G(|x|) \d x = \int_{0}^{\infty}  
\d s 
	\frac{1}{s^3} \int_{B_{1 \wedge s}} G(|x|) \d x \\ &\geq \int_0^1 
	\frac{1}{\psi(1/s)s^2}\,\frac{\d s}{s} + \psi(1) \geq \int_1^{\infty} 
	\frac{u^2}{\psi(u)}\,\frac{\d u}{u} \geq \int_1^{\infty} 
\,\frac{\d u}{u}=\infty,
	\end{align*}
	which is a contradiction. Now let $d=2$. By the same argument
	\begin{align*}
	\int_0^{1/2} G''(t)t \d t = G'(t)t\Big|_0^{1/2}-\int_0^{1/2} G'(t) \d t
	\end{align*}
	and we conclude that the integral is finite. Hence, $\lim_{t \to 
0^+}G(t)<\infty$. By \cite[Theorems 41.5 and 41.9]{MR1739520} we get the 
contradiction.  Finally, for $d=1$ we get that $\lim_{t \to 0^+}G'(t)<\infty$. 
It follows that $\limsup_{t \to 0^+}G(t)/t<\infty$. Due to \cite[Theorem 
$16$]{MR1406564} and \cite[Lemma 2.14]{MR3636597} we obtain that
	\begin{align*}
	\liminf_{x\to\infty}\psi(x)/x^2>0,
	\end{align*}
	which is a contradiction, since $\limsup_{x\to\infty}\psi(x)/x^2=0$.
\end{proof}
\begin{lem}\label{lem:conv_fact}
 Let $D$ be  bounded open and $k \in \N$. If $g \in 
C_{\operatorname{loc}}^k(\Rd \setminus \{ 
0\}) \cap \Ll^1$ and $f \in C^k(D)$ then $g \ast f \in 
C_{\operatorname{loc}}^k(D)$
\end{lem}

\begin{proof}
 Fix $x_0 \in D$. Let $l=\delta_D(x_0)$. Let $\chi_1, \chi_2 \in 
C^{\infty}(\Rd)$ be such that $\ind_{B(x_0,l/4)} \leq \chi_1 \leq 
\ind_{B(x_0,l/2)}$ and $\ind_{B_{l/8}^c} \leq \chi_2 \leq 
\ind_{B_{l/16}^c}$. Observe that $g \ast f = g \ast (f \chi_1) + (g \chi_2) 
\ast (f(1-\chi_1))$ on $B\left(x_0,l/8\right)$. Since $f \chi_1, g \chi_2 \in 
C^k(\Rd)$, it follows that $g \ast f \in C^k(B(x_0,l/8))$. Since $x_0$ was 
arbitrary, the claim follows by induction.
\end{proof}
A consequence of \autoref{lem:conv_fact} is the following corollary.
\begin{cor}\label{cor:ind_fact} Let $D$ be open and bounded and $k \in \N$. If 
$g \in C_{\operatorname{loc}}^k(\Rd \setminus\{ 0\}) \cap \Ll^1$ then $g \ast 
\ind_D \in C_{\operatorname{loc}}^k(D)$.
\end{cor}

The following lemma is crucial in one of the proofs.  

\begin{lem}[{\cite[Proposition 
3.2]{MR3729529}}]\label{lem:harmonic_radial_kernel}
 	Let $X_t$ be an isotropic unimodal L\'{e}vy process in $\Rd$. For every 
$r>0$ there is a radial kernel function $\overline{P}_{r}(z)$ 
 	and a constant $C(r)>0$ such that  $\overline{P}_{r}(z)=C(r)$ for $x 
\in B_r $, $0 \leq \overline{P}_{r}(z) \leq C(r)$ for $z \in \Rd$ and the 
profile function of $\overline{P}_{r}$ is non-increasing. Furthermore, 
if $f$ has the mean-value property in $B(x,r)$
	\begin{align*}
	f(x)=\int_{\Rd} f(z)\overline{P}_{r}(x-z)\d z = f \ast 
\overline{P}_{r}(x).
	\end{align*}
\end{lem}

\section{Weak solutions} \label{sec:weak-solutions}

The aim of this section is to prove \autoref{thm:weak_thm}. For the 
fractional Laplacian related results are known, cf. \cite[Section 
$3$]{MR1671973}. A similar result has recently been obtained in 
\cite{MR3461641} using purely analytic methods instead of probabilistic ones 
exploited in \cite{MR1671973}. When the generalization of these results to more 
general nonlocal operators is immediate, we omit the proof. 

\begin{lem}\label{lem:prob_harmonic_to_anihilation}
	Suppose $u \in \Lspace(\Rd)$ has the mean-value property inside $D$ with 
respect to $X_t$. Then $\LL 
	u=0$ in $D$ in distributional sense.
\end{lem}
\begin{proof}
	Let $\varphi \in C_c^{\infty}(D)$ and $\phi_{\epsilon}$ be a standard 
mollifier 
	(i.e. $\phi_{\epsilon} \in C^{\infty}(\Rd)$ and $\supp \phi_{\epsilon} 
= 
	\overline{B}_{\epsilon}$). Using \eqref{growth_condition} it is easy to 
check that $\phi_{\epsilon} \ast u \in \Lspace(\Rd) \cap C^{\infty}(D)$. Hence, 
$\LL(\phi_{\epsilon} \ast u)$ can be calculated 
	pointwise for $x \in D$ and we have
	\begin{align*}
	(\phi_{\epsilon} \ast u, \LL \varphi) = (\LL (\phi_{\epsilon} \ast u), 
\varphi).
	\end{align*} We consider Dynkin characteristic operator $\UU$. Since it 
is an 
	extension of $\LL$ and is translation-invariant, we obtain
	\begin{align*}
	\LL (\phi_{\epsilon} \ast u) = \UU (\phi_{\epsilon} \ast u) = 
\phi_{\epsilon} 
	\ast \UU u.
	\end{align*}
	We have $\UU u(x)=0$ for $x \in D$, hence
	\begin{align*}
	0=(\LL(\phi_{\epsilon} \ast u), \varphi)=(\phi_{\epsilon} \ast u, \LL 
\varphi), \quad \varphi \in C_c^{\infty}(D_{\epsilon}),
	\end{align*}
	where $D_{\epsilon}=\{ x \in D: \delta_D(x)>\epsilon \}$. Passing 
$\epsilon \to 0$ we get the claim.
\end{proof}

The following lemma is a generalization of \cite[Theorem 
3.9 and Corollary 3.10]{MR1671973}, where the fractional Laplace operator 
is considered.
\begin{lem}\label{lem:anihilation_to_prob_harmonic_C2}
	Let $u \in \Lspace(\Rd) \cap C^2_{\operatorname{loc}}(D)$ be a solution 
of $\LL u=0$ in $D$ in distributional sense. Then $u$ has the mean-value 
property inside $D$.
\end{lem}
\begin{proof}
Since $u \in	\Lspace(\Rd) \cap 
C^2_{\operatorname{loc}}(D)$, $\LL u(x)$ 
can be calculated pointwise for $x \in D$. Fix $D_1 \dsubset D$ and define 
$\widetilde{u}(x)=P_{D_1}[u](x)$, $x 
	\in \Rd$. By the strong Markov property we may assume that $D_1$ is a 
Lipschitz domain. We claim that $\widetilde{u}$ has the mean-value property in 
$D_1$. 
	Indeed, let $D_2$ be an open set relatively compact in $D$ such that 
	$\overline{D_1} \subset D_2$. There exist functions $u_1$, $u_2$ 
on $D_1^c$ such that $u=u_1+u_2$, $u_1$ is continuous and bounded on $D_1^c$ 
and $u_2 \equiv 0$ in $D_2$. We have
		\begin{align*}
	\widetilde{u}(x) = P_{D_1}[u_1](x)+P_{D_1}[u_2](x), 
	\quad x \in \Rd.
	\end{align*}
The first integral is clearly absolutely 
convergent. We claim that it is also continuous as a function of $x$ in 
$\overline{D_1}$. Indeed,  by \autoref{lem:harmonic_radial_kernel} it is 
continuous in $D_1$. Let $x_0 \in \partial D$. For $\epsilon>0$ there exists 
$\delta>0$ such that

\begin{align*}
\lv \int_{D_1^c} P_{D_1}(x,z)u_1(z) \d z - u_1(x_0) \rv \leq \epsilon + 
\norm{u_1}_{\infty} \P^x \(\lv X_{\tau_{D_1}}-x_0 \rv >\delta \) .
\end{align*}   

Since the second term goes to $0$ as $x \to x_0$ (see \cite[Lemmas 2.1 and 
2.9]{MR3350043}), by arbitrary choice of $\epsilon$ we get the claim.
	
\medskip
	
Furthermore, from 
	monotonicity of $1 \wedge \nu^*(h)$ we obtain
	\begin{align*}
	P_{D_1}(x,z) \leq \big(1 \wedge \nu^*(\dist(z,D_1)) \big) \E^x 
\tau_{D_1}, 
	\quad x \in D_1, \ z \in D_2^c.
	\end{align*}
	
	Since $u \in \Lspace(\Rd)$, \eqref{measure_scaling} implies the 
absolute 
convergence of the second integral. Since by \cite[Lemma $2.9$ and Remark 
$2$]{MR3350043} $\E^x \tau_{D_1} \in C_0(D_1)$, it is continuous as well. Hence 
$\widetilde{u}$ is continuous and has the mean-value
property in $D_1$. Note that  $\widetilde{u}=u$ on $D_1^c$, since $D_1$ is a 
Lipschitz domain.

\medskip
	
Let $h=\widetilde{u}-u$. We now verify that $h \equiv 0$ so 
that $u=\widetilde{u}$ has the mean-value property in $D_1$. Since $\LL u=0$ 
in $D_1$, from \autoref{lem:prob_harmonic_to_anihilation} we have $\LL 
h(x)=0$ for $x \in D_1$. Observe $h$ is continuous and compactly supported . 
Suppose it has a  positive maximum at $x_0 \in D_1$, then
	
\begin{align*}
0=\LL h(x_0)=  \int_{\Rd}\(h(y)-h(x_0)\)\nu(x_0-y)\d y ,
\end{align*}

which implies that $h$ is constant on $\mathrm{supp}(\nu)+x_0$. If 
$D_1\subset \mathrm{supp}(\nu)+x_0$ we get that $h\leq 0$. If not we can use the 
chain rule to get for any $n\in\N$ that $h$ is constant on 
$n\mathrm{supp}(\nu)+x_0$ and consequently $h\leq 0$. Similarly, $h$ must be 
non-negative.
\end{proof}

\begin{lem}\label{lem:anihilation_to_prob_harmonic}
	Let $u \in \Lspace(\Rd)$ be a solution of $\LL u=0$ in $D$ in 
distributional sense. 
	Then $u$ has the mean-value property inside $D$.
\end{lem}

\begin{proof}
	Let $\Omega \dsubset D$ be a bounded Lipschitz domain. By \cite{MR1825650} 
and the 
	Ikeda-Watanabe formula we have that the harmonic measure $P_\Omega(x,\d z)$ 
is 
	absolutely continuous with respect to the Lebesgue measure. Define $\rho = 
(1 \wedge \dist (\Omega,D^c))/2$ and let $V=\Omega+B_{\rho}$. For $\epsilon < 
\rho/2$ we consider standard mollifiers $\phi_{\epsilon}$ (i.e. 
$\phi_{\epsilon} \in C^{\infty}(\Rd)$ and $\supp \phi_{\epsilon} = 
\overline{B}_{\epsilon}$). 
Since $\LL$ is translation-invariant we have that $\LL(\phi_{\epsilon} \ast 
u)= \LL u \ast \phi_{\epsilon}=0$ in $V_{\epsilon}=\{ x \in D: 
\dist(x,V^c)>\epsilon \}$ 
	in distributional sense. By 
\autoref{lem:anihilation_to_prob_harmonic_C2} we 
	obtain
	\begin{align*}
	\phi_{\epsilon} \ast u(x) = P_\Omega [ \phi_{\epsilon} \ast u] (x), \quad x 
	\in \Omega.
	\end{align*}
	Note $u \in \Ll^1$ implies  $\phi_{\epsilon} \ast u \to u$ in 
$\Ll^1$. 
	Hence, up to the subsequence
	\begin{align*}
	\lim\limits_{\epsilon \to 0} \phi_{\epsilon} \ast u(x) = u(x) \quad 
	\text{a. e.}
	\end{align*}
	Moreover, since $\phi_{\epsilon}\ast u$ has the mean-value property in 
	$\overline{V}_{\rho/2}$, by \autoref{lem:harmonic_radial_kernel}
	
	\begin{align*}
	\phi_{\epsilon} \ast u(z)
	= \phi_{\epsilon} \ast u \ast \overline{P}_{r}(z)
	\end{align*}
	
	for a fixed $0<r<\rho/4$. Hence, for any $E \subset \Omega^c$
	\begin{align*}
	P_U[\lv \phi_{\epsilon} \ast u \rv;V_{\rho/2}\cap E](x) &\leq 
\int_{V_{\rho/2} \cap \Omega^c\cap E} \lv \phi_{\epsilon} \ast u(z) \rv 
P_\Omega(x,z) \d z \\ &= \int_{V_{\rho/2}\cap \Omega^c \cap E} \lv 
\phi_{\epsilon} \ast u \ast \overline{P}_{r}(z) \rv P_\Omega(x,z) \d z \\ 
&\leq \int_{B_{\epsilon}} \phi_{\epsilon}(s) \int_{\Rd} \lv u(y) \rv 
\int_{V_{\rho/2} 
\cap \Omega^c \cap E} \overline{P}_{r}(z-y-s)P_\Omega(x,z) \d z \d y \d s. 
	\end{align*}
	Let $c=2\sup_{x \in V}|x|$. Then from boundedness of 
$\overline{P}_{r}$ 
	and local integrability of $u$ we get
	\begin{align*}
	\int_{|y|\leq c} |u(y)| \int_{V_{\rho/2} \cap \Omega^c \cap E} 
	\overline{P}_{r}(z-y-s)P_\Omega(x,z) \d z \d y &\leq C \int_{|y|\leq c} 
|u(y)| \d y \int_{E} P_\Omega(x,z) \d z \\ &\leq C \norm{u}_{\Lspace} \int_{E} 
P_\Omega(x,z) \d z.
	\end{align*}

	Furthermore, for $|y|>c$ we have $|z-y-s|>r$, 
hence $\overline{P}_{r}(z-y-s) \leq P_{B_r}(0,z-y-s)$. From 
\eqref{growth_condition} and monotonicity of the L\'{e}vy measure we get
	\begin{align*}
	P_{B_r}(0,y+s-z) \leq 1 \wedge \nu^*(|y-s-z|-r) \E^x \tau_{B_r} \leq C 
	(1 \wedge \nu^*(|y|)).
	\end{align*}
	Thus,
	\begin{align*}
	\int_{|y|>c} |u(y)| \int_{V_{\rho/2}\cap \Omega^c \cap E} 
	\overline{P}_{r}(z-y-s)P_\Omega(x,z) \d z \d y 
	&\leq C\norm{u}_{\Lspace} \int_E P_\Omega(x,z) \d z.
	\end{align*}
	
	It follows that $\phi_{\epsilon} \ast u$ are uniformly integrable 
with respect to the measure $P_\Omega(x,z)\d z$ in $V_{\rho/2}$. By the Vitali 
convergence theorem
	
\begin{align*}
\lim\limits_{\epsilon \to 0} P_\Omega [ \phi_{\epsilon} \ast u; V_{\rho/2} ](x) 
= P_\Omega [u; V_{\rho/2} ](x).
\end{align*}

It remains to show that $\lim_{\epsilon \to 0} P_\Omega [ \phi_{\delta} \ast 
u; V_{\rho/2}^c ] = P_\Omega [u; V_{\rho/2}^c]$. 
Since $\dist(\Omega,V_{\rho/2}^c)=\rho/2$, by the Ikeda-Watanabe formula

\begin{align*}
P_\Omega [\phi_{\epsilon} \ast u;V^c_{\rho/2}](x) &= 
\int_{V_{\rho/2}} \phi_{\epsilon} \ast u(z) \int_\Omega G_\Omega(x,y) \nu(z-y) 
\d z \d y \\ &= \int_{B^c_{\rho/2}} \nu(z) \d z \int_\Omega \phi_{\epsilon} 
\ast u(z+y) 1_{V^c_{\rho/2}}(z+y)G_\Omega(x,y) \d y.
\end{align*}

Using the fact that $\int_\Omega G_\Omega(x,y) \d y=\E^x 
\tau_\Omega<\infty$, $\nu(B_{\rho/2}^c)<\infty$ and $\lim_{\delta \to 0} 
\phi_{\delta} \ast u = u$ in $\Lspace(\Rd)$ we obtain

\begin{align*}
\lim_{\delta \to 0} P_\Omega [ \phi_{\delta} \ast u; V_{\rho/2}^c ] = 
P_\Omega [u; V_{\rho/2}^c].
\end{align*}
Thus $u(x) = P_\Omega[u](x)$ for a. e. $x \in \Omega$.
\end{proof}

Combining \autoref{lem:prob_harmonic_to_anihilation} and 
\autoref{lem:anihilation_to_prob_harmonic} we obtain a following result.

\begin{thm}\label{thm:prob:harmonic_anihilation}
Let $D$  be an open set and $u \in \Lspace$. Then $u$ has the mean-value 
property inside $D$ if and only if $\LL u = 0$ in distributional sense.
\end{thm}

\begin{lem}\label{lem:Gd_weak_solution}
Let $D$ be a bounded open set and $f \in L^1(D)$. Then $-G_D[f]$ is 
a distributional solution of \eqref{eq:weak_problem} with $g \equiv 0$. 
\end{lem}

\begin{proof}
	First assume $f$ is continuous. Then by \cite[Chapter V]{MR0193671} 
we have
\begin{align*}
\UU G_D[f](x)=-f(x), \quad x \in D.
\end{align*}

Let $\phi_{\epsilon}$, $\epsilon>0$, be a standard mollifier. Since $\UU$ is 
an extension of $\LL$ and is translation-invariant we get

\begin{align*}
\LL (\phi_{\epsilon} \ast G_D[f]) = \UU (\phi_{\epsilon} \ast G_D[f]) 
= \phi_{\epsilon} \ast \UU G_D[f] = -\phi_{\epsilon} \ast f. 
\end{align*}

Thus
\begin{align*}
(-\phi_{\epsilon} \ast G_D[f],\LL \phi) = (\phi_{\epsilon} \ast f,\phi).
\end{align*}
Passing $\epsilon \to 0$ we obtain

\begin{align*}
(-G_D[f], \LL \varphi) = (f,\varphi), \quad \varphi \in C_c^{\infty}(\Rd).
\end{align*}

In general case, since $D$ is bounded, we have $\norm{G_D[f]}_{L^1} 
\leq \norm{G_D [1]}_{\infty} \norm{f \ind_D}_{L^1}$ and

\begin{align*}
\norm{G_D[1]}_{\infty} = \sup\limits_{x \in \Rd} G_D[1](x) = 
\sup\limits_{x \in \Rd} \E^x \tauD \leq \E^0 \tau_{B(0, \diam (D))} 
< \infty.
\end{align*}

Using mollification of $f$ we get the claim.
\end{proof}

\begin{proof}[\bf{Proof of \autoref{thm:weak_thm}}]
	Let $h = u + G_D[f]$. By \autoref{lem:Gd_weak_solution} $h$ is 
a harmonic function in distributional sense. Hence, 
by \autoref{lem:anihilation_to_prob_harmonic} $h$ has the mean-value 
property, which finishes the first claim.

\medskip
	
Now let $f, g \ast \nu \in \KK(D \setminus \overline{V})$ and  $D$ be a 
Lipschitz domain. Then it follows that

\begin{align*}
	\widetilde{u}(x) = -G_D[f](x) + P_D[g](x).
	\end{align*}
is a solution of \eqref{eq:weak_problem}, which is bounded near to the 
boundary. Let 
$U_n \nearrow D$ be a sequence of Lipschitz domains approaching $D$. We have
	\begin{align*}
	P_{U_n}[h](x) = P_{U_n}[h;D^c](x) + P_{U_n}[h;\overline{D}\setminus 
U_n](x).
	\end{align*}
	
By the dominated convergence theorem $P_{U_n}[h;D^c](x) \xrightarrow[]{n 
\to \infty} P_D[h;D^c](x)=P_D[g](x)$. Note that by our additional assumptions 
on $g$ and $\nu$ we have that $P_D[g]$ is well-defined. Furthermore, since $f 
\in \KK(D \setminus \overline{V})$, there exists $n_0 \in \N$ such that for $n 
\geq n_0$ we have $V \subset U_n$. From boundedness of $u$ 
and \autoref{lem:Kato_green_op_bounded} we get that $h$ is bounded in 
$\overline{D} \setminus U_n$ for $n > n_0$ and 

\begin{align*}
P_{U_n}[h;\overline{D}\setminus U_n](x) \leq C P_{U_n}\(x, \overline{D} 
\setminus U_n\)\,.
\end{align*}

By \cite[Theorem 1]{MR1825650} we have 
\begin{align*}
\P_{U_n} ( x, \overline{D} 
\setminus U_n ) \xrightarrow[]{n \to \infty} P_D \(x, \partial D\)=0
\end{align*}
	Hence, $u=\widetilde{u}$.
\end{proof}

\section{The sufficient condition for twice 
	differentiability}\label{sec:sufficient_condition}

In this section, we provide auxiliary technical results and the proof of 
\autoref{thm:main_thm}. Throughout this section we assume $D \subset \R^d$ be 
an open bounded set. The following lemmas are modifications of Lemma 2.2 and 
Lemma 2.3 in \cite{MR521856}.

\begin{lem}\label{lem:lemma22modification}
Suppose $f$ is a uniformly continuous function on $D$ and $H(x,y)$ is a 
continuous function for $x,y \in D$, $x \neq y$ satisfying

\begin{align*}
|H(x,y)| \leq F(|x-y|), \quad \lv \frac{\partial H(x,y)}{\partial x_i} \rv \leq 
\frac{F(|x-y|)}{|x-y|}, \quad i=1,...,d 
\end{align*}

for some non-increasing function $F\!:(0,\infty) \mapsto [0,\infty)$. If the 
following holds

\begin{align}\label{int_condition}
	\int_0^{1/2} F(t)\omega_f(t,D)t^{d-1}\d t < \infty,
\end{align}

then the function $g(x) = \int_D H(x,y) \left(f(y) - f(x)\right) \d y$ is 
uniformly continuous in $D$.
\end{lem}

\begin{rem}\label{rem:omega*_remark}
The integral condition \eqref{int_condition} and boundedness of the 
integrand for $1/2 \leq t \leq \diam(D)$ imply that 

\begin{align*}
	\int_0^{\diam(D)} F(t)\omega_f(t,D)t^{d-1} \d t < \infty.
\end{align*}

Moreover,
\begin{align*}
\lim_{h \to 0} h \int_h^{\diam(D)} F(t) \omega_f(t,D)t^{d-2} \d t = 0.
\end{align*}

Indeed, clearly we have
\begin{align*}
	h \int_h^{\diam(D)} F(t) \omega_f(t,D)t^{d-2} \d t = \int_0^{\diam(D)} 
\ind_{[h,\infty)}(t) F(t) \omega_f(t,D)t^{d-1} \frac{h}{t} \d t.
\end{align*}

Since $\ind_{[h,\infty)}(t)h/t\leq 1$, the claim follows by the 
dominated convergence theorem. 
\end{rem}

\begin{proof}
	First note that by integration in polar coordinates one can check that 
	the integral defining $g$ actually exists. Set $\epsilon > 0$. Let $0<h 
< \delta(D)$ and $x$ i $z$ be arbitrary fixed points in $D$ such that  
$|x-z|=h$. Denote $j(x,y):=H(x,y)\( f(y)-f(x) \)$. Observe that 
$|g(x)-g(z)|$ is bounded by the sum of two integrals $I_1$ and $I_2$ of 
$j(x,\cdot)-j(z,\cdot)$ over the sets $D \cap B(x,2h)$ and $D \setminus B(x,2h)$ 
respectively. On $D \cap B(x,2h)$ we have
	
\begin{align*}
	I_1 &= \left\lvert \int_{D \cap B(x,2h)} H(x,y) \( f(y)-f(x) \)  
\d y 
	- \int_{D \cap B(x,2h)} H(z,y) \( f(y)-f(z) \)  \d y 
\right\rvert \\ 
	&\leq  \int_{D \cap B(x,3h)} \left\lvert H(x,y) \right\rvert 
\left\lvert 
	f(y)-f(x) \right\rvert \d y +\int_{D \cap B(z,3h)} \left\lvert H(z,y) 
	\right\rvert \left\lvert f(y)-f(z) \right\rvert  \d y \\ &\leq 
2 \int_0^{3h} 
	F(t)\omega_f(t,D)t^{d-1}\d t < \frac{\epsilon}{3}
	\end{align*}
	
	for sufficiently small $h$. Obviously $I_2\leq I_3+I_4$, where
	
	\begin{align*}
	I_3 &:= \left\lvert \int_{D \setminus B(x,2h)} \( f(y) - f(z) \)
	\( H(x,y) - H(z,y) \) \d y \right\rvert, \\ 
	I_4 &:= \left\lvert f(z) - f(x) \right\rvert \left\lvert \int_{D 
\setminus 
		B(x,2h)} H(x,y) \d y \right\rvert.
	\end{align*}
	
	By the mean value theorem
	\begin{align*}
	I_3 \leq |x-z| \sum_{i=1}^d \int_{D \setminus B(x,2h)} \left\lvert 
	H_{x_i}(\widetilde{x},y) \right\rvert \left\lvert f(y)-f(z) \right\rvert 
 \d y
	\end{align*}
	for some $\widetilde{x} = \theta x + (1-\theta)z$, $\theta \in (0,1)$. 
Note 
	that for $y \in D \setminus B(x,2h)$ we have $|x-y|\geq 2|x-z|=2h>0$. It 
follows 
	that $|\widetilde{x}-y| \geq h$ and consequently 
$|z-y| \leq |z-\widetilde{x}|+|\widetilde{x}-y|\leq 2|\widetilde{x}-y|$.
	
	Thus,
	\begin{align*}
	I_3 &\leq Ch \int_{D \setminus B(x,2h)} 
	\frac{F(|\widetilde{x}-y|)}{|\widetilde{x}-y|} \lv f(y)-f(z) 
	\rv  \d y \\ &\leq Ch \int_{D \setminus B(x,2h)} 
	\frac{F\left(|z-y|/2\right)}{|z-y|} \left\lvert f(y)-f(z) \right\rvert  
\d y 
	\\ &\leq h \int_{D \setminus B(z,h)} 
\frac{F\left(|z-y|/2\right)}{|z-y|} 
	\left\lvert f(y)-f(z) \right\rvert  \d y \leq h \int_{h}^{\diam(D)} 
	F\left(t/2\right) \omega_f(t,D) t^{d-2}\d t \\ &\leq h 
	\int_{h/2}^{\diam(D)/3} F(t) \omega_f(2t,D) t^{d-2}\d t.
	\end{align*}

Thus, by \autoref{rem:omega*_remark} we see that $I_3<\epsilon/3$ for 
sufficiently small $h$. Finally, \eqref{int_condition} implies

\begin{align*}
I_4 &\leq \omega_f(h,D) \int_{D \setminus B(x,2h)} F(|x-y|) \d y =
 \int_0^{\diam(D)} \ind_{[2h,\infty)}(t) F(t) 
\frac{\omega_f(h,D)}{\omega_f(t,D)} \omega_f(t,D) t^{d-1}\d t.
\end{align*}

Observe that $\ind_{[2h,\infty)}(t) \frac{\omega_f(h,D)}{\omega_f(t,D)} \leq 1$ 
by monotonicity of $\omega_f(\cdot,D)$. Thus, \eqref{int_condition} justifies 
the application of the dominated convergence theorem and we obtain

\begin{align*}
\lim_{h \to 0} \omega_f(h,D) \int_{D \setminus B(x,2h)} F(|x-y|) \d y = 0.
\end{align*}

In particular, $I_4\leq \epsilon/3$ for sufficiently small $h$.	It follows that 
$\lv g(x)-g(z) \rv < \epsilon$, if $h$ is sufficiently small. Thus, $g$ is 
uniformly continuous.
\end{proof}

\begin{lem}\label{lem:lemma23modification}
	Suppose $f$ is a uniformly continuous function on $D$ and $H(x,y)$ is a 
continuous 
	function for $x,y \in D$, $x \neq y$ such that $\int\limits_D H(x,y) \d 
y$ is 
	continuously differentiable with respect to  $x$. Assume there exists a 
	non-increasing function $F\!:(0,\infty) \mapsto [0,\infty)$ such that for 
$i,j=1,...,d$
	\begin{align}\label{diff_properties}
	|H(x,y)|, \lv \frac{\partial 
		H(x,y)}{\partial x_i} \rv \leq F(|x-y|), 
\quad 
	\lv \frac{\partial^2 H(x,y)}{\partial x_i \partial x_j} 
	\rv \leq \frac{F(|x-y|)}{|x-y|}.
	\end{align}
If the following holds 
	\begin{align}\label{lemma23condition}
	\int_0^{1/2} F(t)\omega_f(t,D)t^{d-1}\,dt < \infty,
	\end{align}
	then $u(x) = \int_D H(x,y)f(y) \d y$ is continuously 
differentiable 
	with respect to $x \in D$ and
	\begin{align}\label{main_lem_thesis}
	\frac{\partial u(x)}{\partial x_i} = \int_D \frac{\partial 
H(x,y)}{\partial 
		x_i} \( f(y)-f(x) \) \d y + f(x) 
\frac{\partial}{\partial x_i} \int_D 
	H(x,y) \d y, \quad x \in D, \quad i=1,...,d.
	\end{align}
\end{lem}
\begin{proof}
	Fix $s>0$. Let $V_s=\{ x \in D: \dist(x,\partial D) \geq s \}$. We will 
show 
	that \eqref{main_lem_thesis} holds for $x \in B(\overline{x},r)$, where 
$r>0$ 
	is such that $B(\overline{x},4r)\subset V_s$. For $\epsilon < r$ we 
consider standard mollifiers $\phi_{\epsilon}(x)$ and set $f_{\epsilon}(x)= 
\phi_{\epsilon} \ast f$. Note that
\begin{align}\label{modulus_inequality}
\omega_{f_{\epsilon}}\left( h, B(\overline{x},2r) \right) \leq \omega_f(h,D).
\end{align}

For $x \in D$ we define $u_{\epsilon}(x) = \int_D H(x,y)f_{\epsilon}(y)\d y$. 
From boundedness of $f_{\epsilon}$ we see that the integral defining 
$u_{\epsilon}$ is well 
defined and by the dominated convergence theorem  $u_{\epsilon}(x) \to u(x)$ 
for $x \in V_s$, as ${\epsilon} \to 0$. By \autoref{lem:lemma22modification} 
applied to $\frac{\partial H(x,y)}{\partial x_i}$ we have that the function
	
\begin{align}\label{continuity_prop}
\int_D \frac{\partial H(x,y)}{\partial x_i}  
 \( f_{\epsilon}(y)-f_{\epsilon}(x) \)
 \d y + f_{\epsilon}(x) \frac{\partial}{\partial x_i} \int_D H(x,y)\d y
\end{align}

is continuous on $V_s$. Let $x \in B(\overline{x},r)$. 
Integrating \eqref{continuity_prop} with respect to $x_i$ from $\overline{x}_i$ 
to $x_i$ we obtain a continuously differentiable function $\Psi_{\epsilon}(x)$ 
with respect to $x_i$ with \eqref{continuity_prop} being its derivative. Denote 
$x=(\tilde{x},x_d)$ and $\overline{x}=(\tilde{x},\overline{x}_d)$, 
where $\tilde{x}=(x_1,...,x_{d-1})$ and $\overline{x}_d$ is fixed. 
the Fubini theorem and interchanging the order of integration yields

\begin{align*}
\Psi_{\epsilon}(x) &= \int_{\overline{x}_d}^{x_d} \( \int_D 
\frac{\partial H(\tilde{x},s,y)}{\partial s} \( 
f_{\epsilon}(y)-f_{\epsilon}(\tilde{x},s) \)  \d y + 
f_{\epsilon}(\tilde{x},s) \frac{\partial}{\partial s} \int_D H(\tilde{x},s,y) 
\d y \) \d s \\ &= 
\int_D H(\tilde{x},s,y) \( 
f_{\epsilon}(y)-f_{\epsilon}(\tilde{x},s) \) \Big|_{\overline{x_d}}^{x_d}  
\d y-\int_D \int_{\overline{x}_d}^{x_d} H(\tilde{x},s,y) 
\frac{\partial}{\partial s} \( f_{\epsilon}(y)-f_{\epsilon}(\tilde{x},s) \) 
 \d s  \d y \\ &+ f_{\epsilon}(\tilde{x},s) \int_D H(\tilde{x},s,y) \d 
y \Big|_{\overline{x}_d}^{x_d} - \int_{\overline{x}_d}^{x_d} 
\frac{\partial 	f_{\epsilon}(\tilde{x},s)}{\partial s} \int_D 
H(\tilde{x},s,y) \d y  \d s = u_{\epsilon}(x) - 
u_{\epsilon}(\overline{x}).
\end{align*} 

Thus, for $x \in B(\overline{x},r)$ the partial derivative 
$\frac{\partial u_{\epsilon}(x)}{\partial x_d}$ exists and is equal to  
\eqref{continuity_prop}. The same argument applies to any $i=1,...,d$. It 
remains to prove that \eqref{continuity_prop} converges uniformly 
to \eqref{main_lem_thesis}, as $\epsilon \to 0$. Since $f_{\epsilon} \to f$ 
uniformly, as $\epsilon \to 0$, it is enough to prove the convergence of first 
integral in \eqref{continuity_prop}. Fix $\delta > 0$. Since 
$\int_0^{\diam(D)} F(t)\omega_f(t, D) t^{d-1}\d t<\infty$, there is  $\gamma > 
0$ such that $\int_0^{\gamma} F(t)\omega_f(t, D) t^{d-1}\d 
t<\delta/4$. \eqref{modulus_inequality} implies
	
\begin{align}\label{proof_inequality}
& \left\lvert \int_{B(x,\gamma)} \frac{\partial H(x,y)}{\partial x_i} 
\( f_{\epsilon}(y)-f_{\epsilon}(x) \) \d y - \int_{B(x,\gamma)} 
\frac{\partial H(x,y)}{\partial x_i} \( f(y) - f(x) \) \d y 
\right\rvert \nonumber \\ 
& \leq 2 \left\lvert \int_{B(x,\gamma)} \frac{\partial H(x,y)}{\partial x_i} 
	\omega_f(|x-y|,D) \d y \right\rvert \leq 2 \int_0^{\gamma} S(t) 
\omega_f(t,D) 
	t^{d-1}\d t < \frac{\delta}{2}.
	\end{align}
	On the complement of $B(x,\gamma)$ the function $\left\lvert 
\frac{\partial 
		H(x,y)}{\partial x_i} \right\rvert$ is bounded by some constant 
$C>0$. Choose $\epsilon_0>0$ such that $\norm{f_{\epsilon} -f}_{\infty} \leq 
\delta/(4 C |D|)$ for $\epsilon<\epsilon_0$. Then

\begin{align*}
\left\lvert \int_{D \setminus B(x,\gamma)} \frac{\partial (x,y)}{\partial 
x_i} \( f_{\epsilon}(y)-f_{\epsilon}(x)-f(y)+f(x) \) \d y \right\rvert 
\leq 2 \frac{\delta}{4 C |D|} |D| C = \frac{\delta}{2},
\end{align*} 

which combined with \eqref{proof_inequality} and arbitrary choice of $\delta$ 
ends the proof.
\end{proof}

Now we are ready to prove \autoref{thm:main_thm}.

\begin{proof}[\textbf{Proof of \autoref{thm:main_thm}}]
	Let $u$ be of the form
	\begin{align*}
	u(x) &= -G_D[f](x) + P_D[g](x) \\
	 &= -\int_D G(x,y)f(y) \d y + 
	\int_D \E^x G(X_{\tauD},y)f(y) \d y + P_D[g](x) \\ &=: I_1(x) + 
I_2(x) + 
	I_3(x).
	\end{align*}
	Observe $I_3$ has the mean-value property in $D$, thus, by 
\autoref{rem:PDg_Lspace} and \autoref{lem:harm_c2} it belongs to 
$C^2_{\operatorname{loc}}(D)$. Moreover, for $x \in D$ from symmetry of $G$ and 
(G) we obtain that both $G$ and its first and second derivative are bounded 
either by $S(\delta_D(x))$ or $S(\delta_D(x))/\delta_D(x)$, depending on the 
finiteness of $\int_0^{1/2}|G'(t)|t^{d-1}\d t$, and we are allowed to 
differentiate under the integral sign. Hence, it is enough to prove that 
$g(x):=\int_D G(x,y)f(y) \d y$ is in $C^2_{\operatorname{loc}}(D)$. Fix $i,j 
\in \{1,...,d \}$. Consider two cases.

\begin{enumerate}
		\item
		Let $\int_0^1 |G'(t)|t^{d-1}\d t=\infty$. Fix $x \in D$. From 
		\autoref{lem:conv_fact} we get

\begin{align*} \frac{\partial}{\partial x_i} g(x) &= \int_{\Rd} 
G(x-y) \frac{\partial}{\partial x_i} \left( f \chi_1 \right)(y) \d y + 
\int_{\Rd} \frac{\partial}{\partial x_i} \left(G \chi_2 \right) 
(x-y)\left(f \ind_D\right)(y) \( 1 - \chi_1\)(y) \d y  \\ &=: w_1(x) 
+ w_2(x), 
\end{align*}

where the localization functions $\chi_1$ and $\chi_2$ are chosen in dependence 
of $x$. Note that in the integral defining $w_2$, due to the function $\chi_2$ 
and (G), integration w.r.t. $y$ takes place in a region where $G$ and its 
derivative are bounded. Hence, from (G) we see that differentiating under the 
integral sign is justified. We obtain

\begin{align*}
		\frac{\partial}{\partial x_j} w_2(x) = \int_{\Rd} 
\frac{\partial^2}{\partial 
			x_i \partial x_j} \left(G \chi_2 \right) (x-y)\left(f 
\ind_{D}\right)(y) 
		\left[ 1 - \chi_1\right](y) \d y.
		\end{align*}
		If we split $w_1$ into two integrals
		\begin{align*}
		w_1(x) &= \int_{D_1} G(x-y) \frac{\partial}{\partial y_i} \left( 
f \chi_1 
		\right)(y) \d y + \int_{D \setminus D_1} G(x-y) 
\frac{\partial}{\partial y_i} 
		\left( f \chi_1 \right)(y) \d y \\&=: w_3(x) + w_4(x),
		\end{align*}
		where $D_1 \subset D$ is such that $\chi_1 \big\vert_{D_1} 
\equiv 1$ then the 
		same argument can be applied to $w_4$. Thus
		\begin{align*}
		\frac{\partial}{\partial x_j} w_4(x) = \int_{D \setminus D_1} 
		\frac{\partial}{\partial x_j} G(x-y) \frac{\partial}{\partial 
y_i} \left( f 
		\chi_1 \right)(y) \d y.
		\end{align*}
		Next, observe that
		\begin{align*}
		\int_0^{\diam(D_1)}S(t) \omega_{\nabla f}(t,D_1)t^{d-1}\d t 
\leq 
		\int_0^{\diam(D)}S(t) \omega_{\nabla f}(t,D)t^{d-1}\d t < 
\infty. 
		\end{align*}
		Moreover, by \autoref{cor:ind_fact} the function $x \mapsto 
\int_D G(x,y)\,dy$ is continuously differentiable and from (G) we see 
that \eqref{diff_properties} of \autoref{lem:lemma23modification} is 
satisfied for $H(x,y)=G(|x-y|)$ and $F=S$. Hence, for $h(x) = 
\frac{\partial}{\partial x_i} f(x)$ we obtain
		\begin{align*}
		\frac{\partial}{\partial x_j} w_3(x) = \int_D \frac{\partial 
G(x,y)}{\partial 
			x_j} \left( h(y)- h(x) \right)\d y + h(x) 
\frac{\partial}{\partial x_i} \int_D 
		G(x,y)\d y.
		\end{align*}
		\item Now let $\int_0^1 |G'(t)|t^{d-1}\d t < \infty$. In 
this case, by the Fubini theorem and the fundamental theorem of calculus we get
		\begin{align*}
		\frac{\partial}{\partial x_i} \int_D G(x,y)f(y)\d y = \int_D 
\frac{\partial 
			G(x,y)}{\partial x_i} f(y)\d y.
		\end{align*}
		A similar argument applied to $H(x,y) = \frac{\partial 
G(x,y)}{\partial x_i}$ 
		shows that the assumptions of \autoref{lem:lemma23modification} 
are satisfied 
		with $F=S$. Note that here we use the additional assumption 
on $G'''$. Thus,
		\begin{align*}
		\frac{\partial^2}{\partial x_i \partial x_j} \int_D 
G(x,y)f(y)\d y &=  \int_D 
		\frac{\partial^2 G(x,y)}{\partial x_i \partial x_j} \( 
f(y)-f(x) \) \d y + f(x) \frac{\partial}{\partial x_j} \int_D 
\frac{\partial 
			G(x,y)}{\partial x_i}\d y \\ &+ \frac{\partial}{\partial 
x_j} \int_D 
		\frac{\partial G(x,y)}{\partial x_i} f(y)\d y.
		\end{align*}  
	\end{enumerate}
	We have proved that $u \in C^2_{\operatorname{loc}}(D)$. Then by 
\cite[Lemma 4.7]{BGPR2017} the 
	Dynkin characteristic operator $\UU$ coincides with $\LL$. Hence $u$ 
indeed is 
	a solution of the problem \eqref{General_problem3}.
	
	\medskip
	
	Now suppose $\widetilde{u}$ is another solution of 
\eqref{General_problem3}. 
	By \autoref{thm:weak_thm} we find that it is of the form
	\begin{align*}
	\widetilde{u}(x) = - G_D[f](x) + P_U[h](x), \quad x \in U,
	\end{align*}
where $h(x)=u+G_D[f](x)$ and $U$ is any Lipschitz domain such that $U 
\dsubset D$. Fix $x_0 \in D$. Then 
	$U_0=B(x_0,r) \dsubset D$ for any $r < \dist (x_0,D^c)$ and obviously 
$U_0$ is 
	also Lipschitz. Hence,
	\begin{align*}
	\widetilde{u}(x)-u(x) = P_{U_0}[\widetilde{h}](x) - P_D[g](x), \quad x \in U_0,
	\end{align*}
	is harmonic in $U_0$, so it belongs to $C^2_{\operatorname{loc}}(U_0)$. 
The proof yields $-G_D[f] 
	\in C^2_{\operatorname{loc}}(D)$, thus $\widetilde{u}$ is twice 
continuously differentiable in the 
	neighbourhood $x_0$. Since $x_0$ was arbitrary, it follows that every 
solution 
	of \eqref{General_problem3} is $C^2_{\operatorname{loc}}(D)$.
	
\end{proof}

\section{Counterexamples for the case 
	,,\texorpdfstring{$\alpha+\beta=2$}{a+b=2}''}\label{sec:counterexamples}

In this section we provide several counterexamples for \autoref{thm:main_thm}. 
These examples are of the nature ,,$\alpha+\beta=2$'', i.e.,  
for $\alpha \in (0,2)$ we give a function $f \in C^{2-\alpha}(D)$ for 
which the solution of the Dirichlet problem \eqref{Frac_lapl_problem} is not 
twice 
continuously differentiable inside of $D$. In \autoref{sec:examples} we explain 
how the counterexamples can be modified in order to match the assumptions of 
\autoref{thm:main_thm}.

\medskip

Let $D = B_1$. Consider a 
Dirichlet problem
\begin{align}\label{Frac_lapl_problem}
\left\{ \begin{array}{rlll}
\Delta^{\alpha/2} u &=& f & \text{in } D, \\
u &=& 0 & \text{in } D^c, \\
\end{array} \right.
\end{align}
where $\alpha \in (0,2)$. It is known (see \cite{MR1671973} 
or \autoref{thm:weak_thm}) that $u(x) = \int_D G_D(x,y)f(y) \d y$, where 
$G_D(x,y)$ is Green function for the operator $\Delta^{\alpha/2}$ and domain 
$D$ solves \eqref{Frac_lapl_problem}. By the Hunt formula
\begin{align*}
G_D(x,y) = G(x,y) - \E^x G(X_{\tauD},y),
\end{align*} 
where $G$ is the (compensated) potential for process $X_t$ whose generator is 
$\Delta^{\alpha/2}$. Note that since $\E^x G(X_{\tauD},y)$ is $C^{\infty}$, 
the regularity problem is reduced to the regularity of the function $x \mapsto 
g(x) = \int_{B(0,1)}G(x,y)f(y) \d y = G \ast f(x)$.

\subsection{Case \texorpdfstring{$\alpha \in (0,1)$}{a e (0,1)}}
We follow closely the idea from the proof of \autoref{thm:main_thm} apart 
from the fact that at the end we will show that the last function $w_3$ is not 
continuously differentiable. From \autoref{lem:conv_fact} we get
\begin{align}\label{alpha01ref}
\frac{\partial}{\partial x_d} g(x) &= \int_{\Rd} G(x-y) 
\frac{\partial}{\partial y_d} \left( f \chi_1 \right)(y) \d y + \int_{\Rd} 
\frac{\partial}{\partial x_d} \left(G \chi_2 \right) (x-y)\left(f 
\ind_{B_1}\right)(y) \( 1 - \chi_1\)(y) \d y  \nonumber \\ &=: 
w_1(x) + w_2(x),
\end{align}
if only $f \in C_b^1(B_1)$. $\chi_1$ and $\chi_2$ in \eqref{alpha01ref} are 
chosen for $x_0=0$. Put $f(y) = \left((y_d)_+\right)^{2-\alpha}$ and
calculate $\frac{\partial^2}{\partial x_d^2}g(x)$ w $x=0$. Since in $w_2$ we 
are separated from the origin, it follows that 

\begin{align*}
\frac{\partial}{\partial x_d} w_2(x) = \int_{\Rd} \frac{\partial^2}{\partial 
	x_d^2} \left(G \chi_2 \right) (x-y)\left(f \ind_{B_1}\right)(y) 
\( 1 - 
\chi_1\)(y) \d y.
\end{align*}
If we split $w_1$ into
\begin{align*}
w_1(x) = \int_{B_{1/4}} G(x-y) \frac{\partial}{\partial y_d} \left( f \chi_1 
\right)(y) \d y + \int_{B_{1/4}^c} G(x-y) \frac{\partial}{\partial y_d} \left( 
f \chi_1 \right)(y) \d y =: w_3(x) + w_4(x),
\end{align*}
then the same argument applies for $w_4$. Therefore, it remains to calculate the 
derivative of $w_3$. Observe that on ${B_{1/4}}$ we have $f\chi_1 \equiv f$. To 
simplify the notation we accept a mild ambiguity and by $h$ we denote, 
depending on the context, either a real number or a vector in $\Rd$ of the form 
$(0,...,0,h)$. Let $h>0$.

\begin{align*}
\frac{1}{-h} \left( w_3(-h) - w_3(0) \right) &= \frac{2-\alpha}{-h} 
\int_{B_{1/4}} \left( |-h-y|^{\alpha-d} - |y|^{\alpha-d} \right) 
((y_d)_+)^{1-\alpha}  \d y  = \\ &= (2-\alpha) \int_A \frac{|y|^{\alpha-d} - 
	|y+h|^{\alpha-d}}{h} y_d ^{1-\alpha}  \d y=:(2-\alpha)I(h),
\end{align*}
where  $A=B_{1/4} \cap \{y_d>0\}$. 

Let $S_1$ be a $d$-dimensional cube contained in $A$, that is 
\begin{align}\label{S1}
S_1=\{y \in \Rd: |y_i|<a , 0 < y_d < a, \ i=1,...,d-1 \},
\end{align}
where $a=(4\sqrt{d})^{-1}$. Define $S_2 \subset S_1$
\begin{align}\label{S2}
S_2=\{y \in S_1: |y_i| < y_d, \ i=1,...,d-1\}.
\end{align}
By the Fatou lemma and the Fubini theorem
\begin{align*}
\liminf\limits_{h \to 0} I(h) &\geq \int\limits_{A} \liminf\limits_{h \to 0} 
\frac{|y|^{-d+\alpha}-|y_d+h|^{-d+\alpha}}{h} {y_d}^{1-\alpha} \d y  = 
\int\limits_{A} \frac{y_d}{|y|^{d+2-\alpha}} {y_d}^{1-\alpha} \d y  \\ &\geq 
\int\limits_{S_2} \frac{y_d}{|y|^{d+2-\alpha}} {y_d}^{1-\alpha} \d y \geq 
\frac{1}{\sqrt{d}} \int\limits_{S_2} \frac{y_d}{y_d^{d+2-\alpha}} 
{y_d}^{1-\alpha} \d y = C\int\limits_0^a \,\frac{\d y}{y}.
\end{align*}
Hence $\frac{\partial^2}{\partial {x_d}^2} g_- \left(0 \right)=\infty$.
\subsection{Case \texorpdfstring{$\alpha = 1$}{a=1}}\label{alpha1_counterex}
Let $d=1$. The compensated kernel is of the form $G(x,y)=\frac{1}{\pi}\ln 
\frac{1}{|x-y|}$. Note that we cannot apply \cite[Lemma 2.3]{MR521856} because 
(ii) does not hold. Instead write
\begin{align*}
\frac{g(x+h)-g(x)}{h}&=\int_{-1}^{1} 
\frac{G(x+h-y)-G(x-y)}{h} \( f(y)-f(x) \) \d y  \\ &+f(x)\int_{-1}^1 
\frac{G(x+h-y)-G(x-y)}{h} \d y=: I_1(h)+I_2(h).
\end{align*}
Let $f$ be a Lipschitz function. By the mean value theorem

\begin{align*}
\lim_{h \to 0} \int_{-1}^{1} \frac{G(x+h-y)-G(x-y)}{h} \( f(y)-f(x) \) \d 
y = \int_{-1}^1 G'(x-y) \( f(y)-f(x) \) \d y.
\end{align*}
Furthermore, denote

\begin{align*}
F(x)&:=\int_{-1}^1 G(x-y) \d y=-\int_{-1}^1 
\ln{|y-x|} \d y=-\int_{-1-x}^{1-x}\ln |s| \d s \\ &=-\int_0^{1+x}\ln{s} \d s 
-\int_0^{1-x} \ln{s} \d s.
\end{align*}

It follows that
\begin{align*}
\lim_{h \to 0} \int_{-1}^1 \frac{G(x+h-y)-G(x-y)}{h} \d y = F'(x) = 
\ln{\frac{1-x}{1+x}}.
\end{align*}
Hence,
\begin{align}\label{alpha1ref}
g'(x) = \int_{-1}^1 G'(x-y) \( f(y)-f(x) \) \d y + f(x) F'(x).
\end{align}
Put $f(y)=y_+ \ln^{-\beta} \left( 1+\left( y^{-1} \right)_+ \right)$, $\beta 
\in (0,1)$. It is easy to check that $f$ is a Lipschitz function. Let $h<0$. 
Since $f(y)=0$ for $y \leq 0$, from \eqref{alpha1ref} we obtain

\begin{align*}
\frac1h \left( g'(h)-g'(0) \right) &=\frac1h \int_0^1 \left( 
\frac{1}{|h-y|}-\frac{1}{|y|} \right) y_+ \ln^{-\beta}  \left( 1+\frac1y 
\right) \d y \\ &=\frac1h \int_0^1 \left( \frac{1}{y-h}-\frac{1}{y} \right) y 
\ln^{-\beta}  \left( 1+\frac1y \right) \d y \\
& = \frac1h \int_0^1 \frac{h}{y(y-h)} y \ln^{-\beta}  \left( 1+\frac1y \right) 
\d y = \int_0^1 \frac{1}{y-h} 
\ln^{-\beta} \left( 1+\frac1y \right) \d y.
\end{align*} 

By the Monotone Convergence Theorem
\begin{align}\label{divergent_integral}
\int_0^1 \frac{1}{y-h} \ln^{-\beta} \left( 1+\frac1y \right) \d y 
\xrightarrow{h \to 0^-} \int_0^1 \frac1y \ln^{-\beta} \left( 1+\frac1y 
\right) \d y.
\end{align}

Since

\begin{align*}
\lim_{y \to 0^+} \frac{\ln \left( 1+\frac1y \right)}{\ln \frac1y} = 1,
\end{align*}

we obtain $g_-''(0) = \infty$. For $d>1$ and $\beta \in (0,1)$ we apply 
\cite[Lemma 
2.3]{MR521856} to the function $f(y)= (y_d)_+ \ln^{-\beta} \left( 
1+\left(y_d^{-1} 
\right)_+ \right)$, and  $G(x,y) = |x-y|^{-d+1}$ in order to 
obtain

\begin{align}\label{burch_lemma_app}
\frac{\partial}{\partial x_d}g(x)=\int_{B_1} \frac{\partial G(x,y)}{\partial 
	x_d} \left[f(y)-f(x) \right] \d y+f(x) \frac{\partial}{\partial x_d} 
\int_{B_1}G(x,y) \d y.
\end{align}
By \autoref{cor:ind_fact} the condition (iii) of \cite[Lemma 
2.3]{MR521856} holds. Denote
\begin{align*}
H(x,y):=\frac{\partial G(x,y)}{\partial x_d} = (1-d) \frac{\left( x-y 
	\right)_d}{|x-y|^{d+1}} = -C \frac{\left( x-y \right)_d}{|x-y|^{d+1}},
\end{align*}
$C>0$. 
Let $h>0$. We calculate the left-sided second partial derivative 
$\frac{\partial^2}{\partial {x_d}^2} g(x)$ in $x=0$.  

Note that some of terms vanish and the remaining limit is
\begin{align*}
\lim_{h \to 0} \frac{1}{-h} \int\limits_{B(0,1)}\left( H(y+h)-H(y) \right) 
f(y) \d y.
\end{align*}
Let $f_1(s)=f((0,...,0,s))$. We have

\begin{align}\label{wild_calc}
&\int_{B_1}\left( H(y+h)-H(y) \right) f(y) \d y = \int_{B_1}\left( 
H(y+h)-H(y) \right) f_1(y_d) \d y \nonumber \\ = &\int_{B_1}\left( H(y+h)-H(y) 
\right) \int_0^{y_d} f_1'(s) \d s \d y = \int_0^1  \d sf_1'(s) \int_{B_1 
	\cap \H_s}\left( H(y+h)-H(y) \right) \d y,
\end{align}

where $\H_s = \left\lbrace y: y_d > s \right\rbrace$. Denote 
$\tilde{y}=(y_1,...,y_{d-1})$. Then

\begin{align}\label{wild_calc2} 
&\int_0^1  \d sf_1'(s) \int_{B_1 \cap \H_s}\left( H(y+h)-H(y) \right) \d y 
\nonumber \\ = &\int_0^1  \d sf_1'(s) \int_{|\tilde{y}|<1}\d \tilde{y} 
\int_s^{\sqrt{1-|\tilde{y}|^2}} \left[ H(y+h)-H(y) \right]  \d y_d \nonumber \\ 
= &\int_0^1  \d sf_1'(s) \int_{|\tilde{y}|<1}\d \tilde{y} \left[ 
\int_{s+h}^{\sqrt{1-|\tilde{y}|^2}+h} H(y) \d y_d - 
\int_s^{\sqrt{1-|\tilde{y}|^2}} H(y) \d y_d \right] \nonumber \\ = &\int_0^1 
\d sf_1'(s) \int_{|\tilde{y}|<1}\d \tilde{y} \left[ \int_s^{s+h} H(y) \d y_d - 
\int_{\sqrt{1-|\tilde{y}|^2}}^{\sqrt{1-|\tilde{y}|^2}+h} H(y) \d y_d \right] 
\nonumber \\ = &\int_0^1  \d sf_1'(s) \int_{|\tilde{y}|<1} \left[ \left( 
G(\tilde{y},s+h) - G(\tilde{y},s) \right) - \left( 
G(\tilde{y},\sqrt{1-|\tilde{y}|^2}+h) - G(\tilde{y},\sqrt{1-|\tilde{y}|^2})  
\right)  \right] \d \tilde{y} \nonumber \\ =&: I_1(h)-I_2(h).
\end{align}

The Dominated Convergence Theorem implies
\begin{align}\label{wild_calc3}
\lim_{h \to 0^+} &\int\limits_0^1  \d sf_1'(s) \int_{|\tilde{y}|<1} \frac{ 
	G(\tilde{y},\sqrt{1-|\tilde{y}|^2}+h) - 
	G(\tilde{y},\sqrt{1-|\tilde{y}|^2})}{-h} \d \tilde{y} \nonumber \\  
=-&\int_0^1 
\d sf_1'(s) \int_{|\tilde{y}|<1} \lim_{h \to 0^+} \frac{ 
	G(\tilde{y},\sqrt{1-|\tilde{y}|^2}+h) - 
G(\tilde{y},\sqrt{1-|\tilde{y}|^2})}{h} 
\d \tilde{y} \nonumber \\ = -&\int_0^1  \d sf_1'(s) \int_{|\tilde{y}|<1} 
H(\tilde{y},\sqrt{1-|\tilde{y}|^2}) \d \tilde{y} = -\int_{B_1} 
H(\tilde{y},\sqrt{1-|\tilde{y}|^2}) f_1'(y_d) \d y.
\end{align}

Note that the function $H$ under the integral sign is bounded on $B_1$. It 
follows that
\begin{align*}
\lim_{h \to 0^+} \frac{I_2(h)}{h} \leq C \int_{B_1} f_1'(y_d) \d y < \infty.
\end{align*}

By the Fatou lemma

\begin{align}\label{wild_calc4}
\liminf_{h \to 0^+} &\int_0^1  \d sf_1'(s) \int_{|\tilde{y}|<1} \frac{ 
	G(\tilde{y},s+h) - G(\tilde{y},s)}{-h} \d \tilde{y} \\
&\geq \int_0^1  \d sf_1'(s) \int_{|\tilde{y}|<1} -H(\tilde{y},s)\d \tilde{y} 
\nonumber = \int_{B_1} -H(y)f_1'(y_d) \d y. 
\end{align}

We have
\begin{align*}
f_1'(s) = \ln^{-\beta} \left( 1+s^{-1} \right) + \frac{\beta}{s+1} 
\ln^{-\beta-1} \left( 1+s^{-1} \right), \quad s >0.
\end{align*}
Thus,

\begin{align*}
\int_{B_1} &\frac{y_d}{|y|^{d+1}} \ln^{-\beta} \left( 1+\left( y_d^{-1} 
\right)_+ \right)  \d y \geq \int_{S_2} \frac{y_d}{|y|^{d+1}} \ln^{-\beta} 
\left( 1+y_d^{-1} \right)  \d y \\
&\geq \int_{S_2} \frac{y_d}{y_d^{d+1}} 
\ln^{-\beta} \left( 1+y_d^{-1} \right)  \d y \geq C \int_0^a \frac1y 
\ln^{-\beta} \left( 1+y^{-1} \right)  \d y.
\end{align*} 

Hence $\frac{\partial^2}{\partial x_d^2}g_-(0)=\infty$. 

\subsection{Case \texorpdfstring{$\alpha \in (1,2)$}{a e 
		(1,2)}}\label{counterex_12}
Let $d=1$. The compensated potential kernel is of the form $G(x,y) = c_{\alpha} 
|x-y|^{\alpha-1}$. From \cite[Lemma 2.1]{MR521856} we have
\begin{align*}
g'(x) = \int_{-1}^1 G'(y-x)f(y) \d y.
\end{align*} 
We count the second derivative $g(x)$ for $|x|<1$. Observe that
\begin{align*}
I_1(x) := \frac{d}{\d x} \int_{\substack{|y|<1, \\ |y-x|>\frac{1-|x|}{2}}} 
G'(y-x)f(y) \d y = \int_{\substack{|y|<1, \\ |y-x|>\frac{1-|x|}{2}}} 
G''(y-x)f(y) \d y.
\end{align*}
Hence $g''(x) = I_1(x)+I_2(x)$, where
\begin{align*}
I_2(x) := \lim_{h \to 0} \int_{|y-x|<\frac{1-|x|}{2}} = 
\frac{G'(y-x-h)-G'(y-x)}{h}f(y) \d y.
\end{align*}
Put $f(y) = (y_+)^{2-\alpha}$. Then
\begin{align*}
I_2(0) = \lim_{h \to 0} \int_{0}^{1/2} \frac{G'(y-h)-G'(y)}{h} y^{2-\alpha} \d 
y.
\end{align*}
We count the left-sided limit. Let $h > 0$.
\begin{align*}
\int_0^{1/2} \frac{G'(y+h)-G'(y)}{-h}y^{2-\alpha} \d y &= C\int_0^{1/2} 
\frac{y^{\alpha -2} - (y+h)^{\alpha -2}}{h}y^{2-\alpha} \d y  \\&= 
C\int_0^{1/2} \frac{1-\left( 1+ h/y \right)^{\alpha -2}}{h} \d y \\ &= 
C\int_0^{1/(2h)} \left(1 - \left(1 + y^{-1} \right)^{\alpha -2} \right) \d y 
\\&= C\int_{2h}^{\infty} \left(1-(1+s)^{\alpha -2} \right) \,\frac{\d s}{s^2}.
\end{align*}
Thus $g''(0_-)=\infty$.
Now let $d>1$. Then $G(x,y) = |x-y|^{-d+\alpha}$. Denote
\begin{align*}
g(x) = \int_{B_1}G(x,y)f(y) \d y,
\end{align*}
where $f(y) = \left( (y_d)_+ \right)^{2-\alpha}$. \cite[Lemma 2.1]{MR521856} 
implies
\begin{align*}
\frac{\partial g(x)}{\partial x_d} = \int_{B_1}\frac{\partial G(x,y)}{\partial 
	x_d}f(y) \d y.
\end{align*}
We follow closely the argumentation from the case $\alpha=1$, $d>1$. We 
introduce the same notation
\begin{align*}
H(x,y):=\frac{\partial G(x,y)}{\partial x_d} = -(d-\alpha) \frac{\left( x-y 
	\right)_d}{|x-y|^{d+2-\alpha}} = -C \frac{\left( x-y 
	\right)_d}{|x-y|^{d+\beta}},
\end{align*}
$C>0$, $\beta:=2-\alpha \in (0,1)$. Let $h>0$. By repeating \eqref{wild_calc} 
--- \eqref{wild_calc4}
we conclude that it remains to calculate
\begin{align*}
\int_{B_1}-H(y)f_1'(y_d) \d y,
\end{align*}
where $f_1$ is the same as for $\alpha=1$. Here the derivative has simpler 
form. Note that the argumentation \eqref{wild_calc} -- \eqref{wild_calc4} is 
correct even though $f_1$ does not 
belong to $C^1(B_1)$ for $\alpha > 1$. We obtain 
\begin{align*}
\int_{B_1} \frac{y_d}{|y|^{d+\beta}}\left(y_d\right)_+^{1-\alpha} \d y &= 
\int_A \frac{y_d}{|y|^{d+\beta}}y_d^{1-\alpha} \d y \geq \int_{S_2} 
\frac{y_d}{|y|^{d+\beta}}y_d^{1-\alpha} \d y \geq \int_{S_2} 
\frac{y_d}{y_d^{d+\beta}}y_d^{1-\alpha} \d y \\ &\geq C \int_0^a 
\,\frac{\d y}{y}=\infty.
\end{align*}
Hence $\frac{\partial^2}{\partial x_d^2}g_-(0)=\infty$.

\section{Examples}\label{sec:examples}

In the last section we present some examples of operators $\LL$ resp. 
corresponding Dirichlet problems that allow for an application of 
\autoref{thm:main_thm}. In 
\autoref{ex:frac-laplace} we modify the considerations from 
\autoref{sec:counterexamples} in order to match the assumptions of 
\autoref{thm:main_thm}. In \autoref{ex:subordinate-BM} we generalize to 
subordinated Brownian motion. Finally, in \autoref{ex:scaling_prop} we extend 
the above class and discuss the process which is assumed only to have the lower 
scaling property on the characteristic exponent. 

\begin{exmp}[fractional Laplace operator]\label{ex:frac-laplace}
	Let $X_t$ be strictly stable process whose generator is the fractional 
Laplace 
	operator $-(-\Delta)^{\alpha/2}$. Let $D$ be a bounded open set.

	\begin{enumerate}
		\item Let $\alpha \in (0,1)$. The potential kernel is of the form 
$G(y) = c_{d,\alpha}|y|^{\alpha-d}$ and satisfies 
\begin{align}\label{ex_Gint}
		\int_0^{1/2} |G'(t)|t^{d-1}\d t = \infty.
\end{align}
Here $S(r)=|G'(r)|$. According to \autoref{thm:main_thm}, there is 
a $C^2_{\operatorname{loc}}(D)$ solution of \eqref{General_problem3} if the 
following holds:
\begin{align}\label{ex_alpha01_cond}
 \int_0^{1/2} |G'(t)|\omega_{ \nabla f}(t,D)t^{d-1}\,dt = 
\int_0^{1/2} t^{\alpha-2} \omega_{\nabla f}(t,D)\,dt < \infty.
\end{align}
Obviously our function from the counterexample $f(y)=((y_d)_+)^{2-\alpha}$ 
which is $C^{2-\alpha}(D)$ does not satisfy \eqref{ex_alpha01_cond}. On the 
		other hand, it is well known that for any function which is 
		$C^{2-\alpha+\epsilon}$, $\epsilon>0$ (i.e. 
		$\tilde{f}(y)=((y_d)_+)^{2-\alpha+\epsilon}$), the solution of 
		\eqref{General_problem3} is $C^2_{\operatorname{loc}}(D)$. 
Clearly, this function satisfies 
		\eqref{ex_alpha01_cond} as well, so in some sense 
\autoref{thm:main_thm} 
		extends already known results. The sufficient condition is 
also 
		$\omega_{\nabla f}(t,D) \leq C t^{1-\alpha} \ln^{-\beta} \left( 
1+ t^{-1} \right)$, $\beta > 1$. Then

\begin{align*}
\int_0^{1/2} |G'(t)|\omega_{\nabla f}(t,D)t^{d-1}\d t &= \int_0^{1/2} 
t^{\alpha-2} t^{1-\alpha} \ln^{-\beta} \left( 1 + t^{-1} \right) \d t 
\\ 
&\leq C\int_0^{1/2} t^{-1} \ln^{-\beta}\left( t^{-1} \right) \d 
t \\ &= C\int_{\ln 2}^{\infty} \,\frac{\! \d t}{t^{\beta}} < \infty.
\end{align*}
		
Calculations in the cases below are very similar and therefore will be omitted.

\item Let $\alpha=d=1$. The compensated potential kernel is of 
the form $G(y) = \frac1\pi \ln \frac{1}{|y|}$ and \eqref{ex_Gint} holds for 
$S(r)=|G'(r)|$. Note that in this case $|G'(r)| \neq c \frac{G(r)}{r}$. 
By \autoref{thm:main_thm} the solution of \eqref{General_problem3} will be 
in $C^2_{\operatorname{loc}}(D)$ if

\begin{align*}
	\int_0^{1/2} |G'(t)|\omega_{\nabla f}(t,D)t^{d-1}\d t = \int_0^{1/2} 
t^{-1} \omega_{\nabla f}(t,D)\d t < \infty.
\end{align*}

Hence, it suffices that $\omega_{\nabla f}(t,D)\leq C \ln^{-\beta} 
\left( 1+ t^{-1} \right)$, $\beta>1$.
\item Let $\alpha =1, d > 1$. 
The potential kernel has a form $G(y) = c_{d,\alpha} |y|^{1-d}$ and 
\eqref{ex_Gint} holds for $S(r)=|G'(r)|$. Analogous to 
the case 
		$\alpha \in (0,1)$ it suffices that $\omega_{ \nabla f}(t,D)\leq 
C \ln^{-\beta} 
		\left( 1+ t^{-1} \right)$, $\beta > 1$.
\item $\alpha \in (1,2), d=1$. The compensated potential kernel 
is of the form $G(y)=c_{\alpha} |y|^{\alpha-1}$, $S(r)=|G''(r)|$, and we  have 
$\int_0^1 |G'(t)|\d t<\infty$, thus by \autoref{thm:main_thm}, there 
will be a $C^2_{\operatorname{loc}}(D)$ solution if
	
\begin{align}\label{ex_alpha12_cond}
		\int_0^{1/2} |G''(t)| \omega_f(t,D) t^{d-1}\d t = \int_0^{1/2} 
t^{\alpha-3} \omega_f(t,D) 
		\d t < \infty.
		\end{align}
		Clearly the function $f(y)=(y_+)^{2-\alpha}$ from 
\autoref{sec:counterexamples} 
		does not satisfy \eqref{ex_alpha12_cond}. In order to correct it 
we must either 
		take a function from $C^{2-\alpha+\epsilon}(D)$, $\epsilon > 0$ 
(i.e. 
		$\tilde{f}(y)=(y_+)^{2-\alpha+\epsilon}$) or a function whose 
modulus of 
		continuity is of the form $\omega_{\tilde{f}}(t,D) = 
t^{2-\alpha} \ln^{-\beta} 
		\left( 1+t^{-1} \right)$, $\beta>1$.
		\item $\alpha \in (1,2), d\geq 2$. The potential kernel has the 
form $G(y) = c_{d,\alpha}|y|^{\alpha-d}$ and $S(r)=|G''(r)|$. We have 
		\begin{align*}
		\int_0^{1/2} |G'(t)|t^{d-1}\d t < \infty.
		\end{align*}
		By \autoref{thm:main_thm} we have to take a function $\tilde{f}$ 
from 
		$C^{2-\alpha+\epsilon}(D)$ or such that its modulus of 
continuity has the form 
		$\omega_{\tilde{f}}(t,D) = t^{2-\alpha} \ln^{-\beta} \left( 1+ 
t^{-1} \right)$, 
		$\beta>1$.
	\end{enumerate}
\end{exmp}

\begin{exmp}[Subordinate Brownian motion]\label{ex:subordinate-BM}
	Let $(B_t, t\geq 0)$ be a Brownian motion in $\Rd$ and $(S_t, t \geq 0)$ 
--- a 
	subordinator independent from $B_t$, i.e. a L\'{e}vy process in $\R$ 
which 
	stars from $0$ and has non-negative trajectories. Process $(X_t, t\geq 
0)$ 
	defined by $X_t=B_{S_t}$ is called a subordinated Brownian motion.
	Denote by $\phi$ the Laplace exponent of $S_t$:
	\begin{align*}
	\E \exp \{ -\lambda S_t \} =  \exp \{ -t \phi(\lambda) \}.
	\end{align*}
	It is well known that $\phi$ is of the form
	\begin{align*}
	\phi(\lambda) = \gamma t + \int_0^{\infty} \left( 1- e^{-\lambda t} 
\right) 
	\,\mu(\! \d t)
	\end{align*}
	where $\mu$ is the L\'{e}vy measure of $S_t$ satisfying $\int_0^{\infty} 
(1 \wedge t) \mu(\! \d t)<\infty$. The corresponding operator is of the 
form $\LL = -\phi(-\Delta)$ and we have $\psi(\xi) = \phi(|\xi|^2)$. An 
example of subordinated Brownian motion is the 
	process from \autoref{ex:frac-laplace} with $\phi(\lambda) = 
	\lambda^{\alpha/2}$, $\alpha \in (0,2)$. Another example is geometric 
stable 
	process with $\phi(\lambda)=\ln \left(1+\lambda^{\alpha/2} \right)$, 
$\alpha \in 
	(0,2)$. Denote by $G_d(r)$ the potential of $d$-dimensional 
subordinated 
	Brownian motion $X_t$. From  \cite[Theorem 
	5.17]{MR3646773} we have
	\begin{align}\label{subord_green_est}
	G_d(r)  \asymp r^{-d-2} \frac{\phi'(r^{-2})}{\phi^2(r^{-2})}, \quad r 
	\to 0^+,
	\end{align} if $d \geq 3$ and there exist $\beta \in [0,d/2+1)$ and 
$\alpha>0$ such that $\phi^{-2}\phi'$ satisfies weak lower and upper scaling 
condition at infinity with exponents $-\beta$ and $-\alpha$, respectively (see 
\cite{MR3646773}). The same result under slightly stronger assumptions is derived 
in 
\cite[Proposition 3.5]{MR2928720}.
	 \medskip 
	For $d$-dimensional subordinated Brownian motion $X_t$, $d\geq 3$, we 
have
	\begin{align*}
	G_d(r) = \int_0^{\infty} (4 \pi t)^{-d/2} \exp \left( -\frac{r^2}{4t} 
\right) 
	\,u(\! \d t).
	\end{align*}
	It follows that
	\begin{align}\label{G'_subordinator}
	G_d'(r) &= G_d(r) = -\int_0^{\infty} (4 \pi t)^{-d/2} \exp \left( 
	-\frac{r^2}{4t} \right) \frac{2r}{4t} u(\! \d t) \nonumber \\ &= -2r \pi 
\int_0^{\infty} (4 
	\pi t)^{-(d+2)/2} \exp \left( -\frac{r^2}{4t}  \right) u(\! \d t) = -2r 
\pi 
	G_{d+2}(r).
	\end{align}
	That and \eqref{subord_green_est} imply
	\begin{align*}
	\left\lvert G_d'(r) \right\rvert \leq Cr \cdot r^{-(d+1)-2} 
	\frac{\phi'(r^{-2})}{\phi^2(r^{-2})} = C \frac1r r^{-d-2} 
	\frac{\phi'(r^{-2})}{\phi^2(r^{-2})} \leq C \frac{G_d(r)}{r}.
	\end{align*}
	By induction
	\begin{align*}
	\left\lvert G_d^{(k)}(r) \right\rvert \leq C \frac{G(r)}{r^k}, \quad k 
	\in \N.
	\end{align*}
	Thus, the necessary conditions involving $G$ and its derivatives hold 
true for $S(r)=G_d(r)/r^2$. Note that the density of 
	L\'{e}vy measure of $X_t$

\begin{align*}
	\nu(r) = \int_{0}^{\infty} (4 \pi t)^{-d/2} \exp \left( -\frac{r^2}{4t} 
\right) \,\mu(\! \d t)
\end{align*}

belongs to $C^{\infty}$. By \cite[Lemma $7.4$]{BGPR2017} the 
assumptions of \autoref{thm:main_thm} are satisfied with $\nu^*\equiv \nu$ if 
$\phi$ is a complete Bernstein function.

	\medskip
	
	Take geometric stable process with $\phi(\lambda)= \ln \left( 
	1+\lambda^{\alpha/2} \right)$. Then by \eqref{subord_green_est} and 
\eqref{G'_subordinator}
	
\begin{align*}
\int_0^{1/2} \lv G_d'(t) \rv t^{d-1}\d t &\geq  C \int_0^{1/2} t^{-d-3} t^{d-1} 
\frac{\phi'(t^{-2})}{\phi^2(t^{-2})}\d r = C \int_0^{1/2} \frac{1}{t^4} 	
\frac{\frac{1}{1+t^{-\alpha}} \frac{1}{t^{\alpha-2}}}{\ln^2\left( 
1+t^{-\alpha} \right)}\d t \\ &\geq C \int_0^{1/2} \frac{1}{t^2} 
\frac{1}{\ln^2\left(1 + t^{-\alpha}\right)}\d t \geq \int_0^{1/2} \frac{1}{t^2 
\ln^2 t^{-1}}\d t = \infty,
	\end{align*}
	
	hence, for the solution of \eqref{General_problem3} to be in 
$C^2_{\operatorname{loc}}(D)$, it 
	suffices that the modulus of continuity of gradient of function $f$ is 
of the form $\omega_{\nabla f}(t,D)=t \ln^{1-\epsilon} \left( 1+ 
t^{-1}\right)$, $\epsilon \in (0,1)$.
\end{exmp}

\medskip

Before moving to the last example, let us define concentration functions $K$ 
and $h$ by setting
\begin{align*}
K(r) = \frac{1}{r^2} \int_{|x| \leq r} |x|^2 \nu (\! \d x), \quad r>0,\\
h(r)=\int_{\Rd} \(1 \wedge \frac{|x|^2}{r^2}\)\nu(\! \d x), \quad r>0.
\end{align*}
\begin{prop}\label{prop:scaling_prop}
	Let $d \geq 3$. Suppose there exist $c>0$ and $\alpha \geq 3/2$ such that
	\begin{align}\label{scaling_cond}
	h(r) \leq c\lambda^{\alpha} h(\lambda r), \quad \lambda \leq 1, r>0.
	\end{align}
	Then there exists $c>0$ such that $|U'(r)|\leq cU(r)/r$, $|U''(r)|\leq cU(r)/r^2$, $|U'''(r)|\leq c U(r)/r^3$ for $r>0$.
\end{prop}
\begin{proof} Observe that for $d \geq 3$ the potential $U$ always exists. By 
\cite[Theorem $3$]{MR3225805} there exists $c>0$ such that
	
	\begin{align*}
	U(x)\geq \frac{c}{|x|^d h(1/r)}, \quad r>0.
	\end{align*}
	
	Our aim is to prove (G). By definition and isotropy of $p_t$ 
	
	\begin{align*}
	U(r)=\int_0^{\infty} p_t(\tilde{r}) \d t,
	\end{align*}
	
	where by $\tilde{r}=(0,...,0,r) \in \Rd$. Since $p_t$ is radially 
decreasing, by the Tonelli theorem
	
	\begin{align*}
	U(r)-U(1) = \int_0^{\infty} \int_1^r \partial_{x_d} p_t(y) \d y \d t = 
\int_1^r \int_0^{\infty} \partial_{x_d} p_t(\tilde{y}) \d t \d y,
	\end{align*}
	
	where $\tilde{y} = (0,...,0,y)\in\Rd$. Hence,
	
	\begin{align*}
	U'(r)=\int_0^{\infty} \partial_{x_d} p_t(\tilde{r}) \d t, \quad r>0.
	\end{align*}
	
	By \cite[Theorem $5.6$ and Corollary $6.8$]{GS2017}
	
	\begin{align*}
	\lv \partial^{\beta}_x p_t(x) \rv \leq c \( h^{-1}(1/t)\)^{-|\beta|} 
\varphi_t(x), \quad t>0, x \in \Rd,
	\end{align*}
	
	where
	\begin{align*}
	\varphi_t(x) = \left\{
	\begin{array}{ll}
	\(h^{-1}(1/t)\)^{-d}, & |x|\leq h^{-1}(1/t), \\
	tK(|x|)|x|^{-d}, & |x|> h^{-1}(1/t).
	\end{array}\right.
	\end{align*}
	Let us estimate $|U'(r)|$. We have
	
	\begin{align*}
	 |U'(r)| \leq \frac{K(|x|)}{|x|^{d}}\int_0^{1/h(|x|)} \frac{t}{ 
h^{-1}(1/t)} \d t + \int_{1/h(|x|)}^{\infty} \frac{\d t}{\(h^{-1}(1/t)\)^{d+1}}.
	\end{align*}
	
	The scaling property of $h$ for $|x| > h^{-1}(1/t)$ yields
	
	\begin{align*}
	h(|x|) \leq c \(\frac{h^{-1}(1/t)}{|x|}\)^{\alpha} h(h^{-1}(1/t)).
	\end{align*}
	
	It follows that
	
	\begin{align*}
	\frac{K(|x|)}{|x|^{d}}\int_0^{1/h(|x|)} \frac{t}{ h^{-1}(1/t)} \d t 
&\leq c \frac{K(|x|)}{|x|^{d+1}}\int_0^{1/h(|x|)} t 
\(\frac{1}{th(|x|)}\)^{1/\alpha} \d t \\ &\leq c \frac{K(|x|)} {|x|^{d+1}} 
(h(|x|))^{-1/\alpha} \int_0^{1/h(|x|)} t^{1-1/\alpha} \d t.
	\end{align*}
	
	For $\alpha > 1/2$ the integral is finite and we get
	
	\begin{align*}
	\frac{K(|x|)}{|x|^{d}}\int_0^{1/h(|x|)} \frac{t}{ h^{-1}(1/t)} \d t 
\leq c \frac{K(|x|)}{|x|^{d+1}h(|x|)^2}.
	\end{align*}
	
	The comparability $K$ and $h$ (\cite[Lemma $2.3$]{GS2017}) implies 
	
	\begin{align*}
	\frac{K(|x|)}{|x|^{d}}\int_0^{1/h(|x|)} \frac{t}{ h^{-1}(1/t)} \d t 
\leq c \frac{1}{|x|^{d+1}h(|x|)} \leq c \frac{U(r)}{r}.
	\end{align*}
	
	Furthermore, we always have $h(r) \geq \lambda^2 h(\lambda r)$ for 
$\lambda \leq 1$ and $r>0$. Thus,
	
	\begin{align*}
	\int_{1/h(|x|)}^{\infty} \frac{\d t}{\(h^{-1}(1/t)\)^{d+1}} &= 
\frac{1}{|x|^{d+1}} \int_{1/h(|x|)}^{\infty} 
\frac{|x|^{d+1}}{\(h^{-1}(1/t)\)^{d+1}}  \d t \\ &\leq \frac{1}{|x|^{d+1}} 
\int_{1/h(|x|)}^{\infty} \(\frac{1}{th(|x|)}\)^{(d+1)/2} \d t. 
	\end{align*}
	
	Since $d>1$, the integral is finite and we get
	
	\begin{align*}
	\int_{1/h(|x|)}^{\infty} \frac{\d t}{\(h^{-1}(1/t)\)^{d+1}} \leq c 
\frac{1}{|x|^{d+1}h(|x|)} \leq c \frac{U(r)}{r}.
	\end{align*}
	
	Hence, for $\alpha>1/2$ we obtain $|U'(r)|\leq cU(r)/r$, $r>0$. By 
similar argument one may conclude that $|U''(r)| \leq cU(r)/r^2$ if $\alpha>1$ 
and $|U'''(r)| \leq cU(r)/r^3$ for $\alpha>3/2$.
\end{proof}

\begin{exmp}\label{ex:scaling_prop}
	Let $d \geq 3$, $\alpha > 3/2$, and $X_t$ be a truncated 
$\alpha$-stable L\'{e}vy process in $\Rd$, i.e. with L\'{e}vy measure $\nu(\! \d 
x)=|x|^{-d-\alpha} \varphi(x)$, where $\varphi$ is a cut-off function, i.e. 
$\varphi \in C^{\infty}(\Rd)$ and $\ind_{B_{1/2}} \leq \varphi \leq 
\ind_{B_1}$. One can easily check that $h(r)\asymp r^{-\alpha} \wedge r^{-2}$. 
\autoref{prop:scaling_prop} yields that the assumptions of 
\autoref{thm:main_thm} imposed on function $G$ are satisfied. Observe that (A) 
and \eqref{growth_condition} is satisfied for $\nu^*\equiv 0$. In that case the 
appropriate $\Lspace$ space is simply $\Ll^1$.
\end{exmp}


\appendix
\section{Potential theory for recurrent unimodal L\'{e}vy 
process}\label{app:appendix}
	In this appendix we establish a formula for the Green function for a 
bounded open set $D$ in case of recurrent unimodal L\'{e}vy process $X_t$. 
Contrary to the transient case, here the potential kernel 
$U(x)=\int_0^{\infty}p_t(x) \d t$ is infinite, so the classical Hunt formula has 
no application. Instead, one can define the $\lambda$-potential kernel 
$U^{\lambda}$ by setting
\begin{align*}
	U^{\lambda}(x)=\int_0^{\infty} e^{-\lambda t} p_t(x) \d t.
\end{align*}

	Similarly, we define the $\lambda$-Green function for an open set $D$
		\begin{align*}
	G_D^{\lambda}(x,y)=\int_0^{\infty} e^{-\lambda t} p_t^D(x-y) \d t.
	\end{align*}

Note that both $U^{\lambda}$ and $G_D^{\lambda}$ exist. An analogue of 
the Hunt formula for $G_D^{\lambda}$ holds, namely, for $x,y \in D$

\begin{align*}
 	G_D^{\lambda}(x,y)=U^{\lambda}(y-x)-\E^x \[ e^{-\lambda \tauD} 
U^{\lambda}(y-X_{\tauD}) \].
\end{align*} 

	\begin{lem}\label{lem:lambdaUlambda_lem}
 		Let $d \geq 1$. For any fixed $x_0 \in \Rd \setminus \seto$ we 
have $\lambda U^{\lambda}(x_0) \to 0$ as $\lambda \to 0$.
	\end{lem}

	\begin{proof}
 	 	In the following part we introduce a mild ambiguity by denoting 
by $1$, depending on the context, either a real number or the vector 
$(0,...,0,1) \in \Rd$. Set $x_0=1$. Let $f_{\lambda}(r)= \int_{|x|<r}\d x 
\int_0^{\infty} e^{-\lambda u}p_u(x) \d u$. We have
	
	\begin{align*}
	L f_{\lambda}(s) &=\int_0^{\infty} e^{-st} f_{\lambda}(t) \d 
t=\int_0^{\infty} e^{-st} \int_{|x|<\sqrt{t}} \d x \int_0^{\infty} e^{-\lambda 
u} p_u(x)\d u \d t \\ &=\int_{\Rd} \d x \int_{t>|x|^2} e^{-st}\d t 
\int_0^{\infty} e^{-\lambda u} p_u(x)\d u =\frac1s \int_0^{\infty} e^{-\lambda 
u} \int_{\Rd} e^{-s|x|^2} p_u(x) \d x \d u.
	\end{align*}
	
	By \cite[Lemma $6$]{MR3225805}
	\begin{align*}
	\int_{\Rd} e^{-s|x|^2} p_u(x) \d x =c_d \int_{\Rd} 
e^{-u\psi(\sqrt{s}x)}e^{-|x|^2/4} \d x.
	\end{align*}
	
	Hence, we have for $\lambda>0$
	\begin{align*}
	s L_{\lambda} f(s)=c_d \int_0^{\infty} e^{-\lambda u} \d u \int_{\Rd} 
e^{-u \psi(\sqrt{s} \xi)}e^{-|\xi|^2/4}  \d \xi = c_d \int_{\Rd} 
\frac{1}{\lambda+\psi(\sqrt{s}\xi)} e^{-|\xi|^2/4} \d \xi.
	\end{align*}
	
	By monotonicity of $f$
	\begin{align*}
f_{\lambda}(r) &= \frac{e}{r} \int_r^{\infty} e^{-u/r} f(r) \d u 
\leq \frac{e}{r} \int_0^{\infty}  e^{-u/r} f_{\lambda}(u) \d u = \frac{e}{r} L 
f_{\lambda}(1/r) \\
&= c' \int_{\Rd} \frac{1}{\lambda+\psi(\sqrt{1/r}\xi)} 
e^{-|\xi|^2/4}\d \xi.
	\end{align*}
	
	Since by \cite[Lemma $1$ and Proposition $1$]{MR3225805}
	
	\begin{align*}
	\sup_{|x|\leq 1} \psi(x) \leq \frac{4}{|\xi|^2}  \sup_{|x|\leq |\xi|} \psi(x) 
\leq c \frac{\psi(\xi)}{|\xi|^2},
	\end{align*}
	we obtain
	
	\begin{align*}
	\lambda G^{\lambda}(1) \leq \lambda \frac{f_{\lambda}(1)}{|B_1|} \leq 
c_d \int_{\Rd} \frac{\lambda}{\lambda+\psi(\xi)}e^{-|\xi|^2/4} \d \xi \leq 
\frac{\lambda}{\psi(1)} \int_{B_1^c} e^{-|\xi|^2/4} \, d\xi + \int_{B_1} 
\frac{\lambda}{\lambda + |\xi|^2}  \d \xi.
	\end{align*}
	
	Hence, $\lambda U^{\lambda}(1) \to 0$ as $\lambda \to 0$. The extension 
to arbitrary $x_0$ is immediate.
	\end{proof} 

\begin{lem}\label{lem:compensation_lemma}
Let $x_0 \in \Rd \setminus \seto$ be an arbitrary fixed point. 
For all $x \in \Rd \setminus \seto$ we have $\int_0^{\infty} \lv p_t(x)-p_t(x_0) 
\rv \, \d t<\infty$.
\end{lem}

\begin{proof}
 Let $f \in C_c^{\infty}(\Rd)$ be such that $\ind_{B_{\epsilon}} \leq f 
\leq \ind_{B_{4\epsilon}}$, where $0<4\epsilon<1$. Denote

\begin{align*}
	W_{x_0}^{\lambda}(x)&=\int_0^{\infty} e^{-\lambda t} 
\(p_t(x)-p_t(x_0)\) \d t, \quad x \neq 0,\\ 
W_{x_0}(x) &= \int_0^{\infty}\(p_t(x)-p_t(x_0)\) 
\d t, \quad x \neq 0.
\end{align*} 

Let $x_0=1$. Observe that
\begin{align*}
W_1^{\lambda} \ast f(0)=\int_0^{\infty} e^{-\lambda t} \( p_t \ast f(0) - 
p_t(1)\norm{f}_1 \) \d t.
\end{align*}

Note that the integrand has a positive sign. Indeed,
\begin{align*}
p_t \ast f(0)-p_t(1)\norm{f}_1 = \int_{B_{4\epsilon}} \(p_t(y)f(y)-p_t(1)f(y)\) 
\d y>0,
\end{align*}

since $4\epsilon<1$. Furthermore,
\begin{align*}
p_t(1)\norm{f}_1 = \int_{B_{4\epsilon}} p_t(1)f(y) \d y \geq 
\int_{B_{4\epsilon}} p_t(1+4\epsilon-y)f(y) \d y = p_t \ast f(1+4\epsilon).
\end{align*}

Hence, by the Fourier inversion theorem
\begin{align*}	
\int_0^{\infty} e^{-\lambda t} \( p_t \ast f(0) - 
p_t(1)\norm{f}_1 \) \d t &\leq \int_0^{\infty} e^{-\lambda t} \int_{\Rd} (1-\cos 
\((1+4\epsilon)\xi \))\widehat{p_t}(\xi)\lv \widehat{f}(\xi) \rv \d \xi \d t  \\ 
&\leq \int_{\Rd} (1-\cos \((1+4\epsilon)\xi \)) \frac{\lv \widehat{f}(\xi) 
\rv}{\psi(\xi)} \d \xi.
\end{align*}

By the monotone convergence theorem and the fact that $ \lv \widehat{f}(\xi) 
\rv$ decays faster than any polynomial
\begin{align*}
 W_1 \ast f(0) = \lim_{\lambda \to 0} W_1^{\lambda} \ast f(0) \leq 
\int_{\Rd} (1-\cos \((1+4\epsilon)\xi \)) \frac{\lv \widehat{f}(\xi) 
\rv}{\psi(\xi)} \d \xi < \infty.
\end{align*}

Hence,
\begin{align}\label{W-L1loc}
\int_{B_{\epsilon}}W_1(x) \d x \leq W_1 \ast f(0) < \infty.
\end{align}

Since $W_1$ is radially decreasing and positive for $|x|<1$, 
\eqref{W-L1loc} implies that it may be infinite only for $x=0$. It follows that 
$W_1$ is well defined for $0<|x|\leq1$. Similarly $0 \leq W_{x_0}<\infty$ for 
$0<|x|\leq |x_0|$.
	
\medskip

It remains to notice that for $|x|>|x_0|$ we have $0 \leq \lv  W_{x_0}(x) 
\rv=-W_{x_0}(x) = W_x(x_0)<\infty$ by the first part of the proof.
\end{proof}

\autoref{lem:compensation_lemma} allows us to introduce, following 
\cite{MR0126885}, \cite{MR0099725}, \cite{MR2256481}, a compensated potential 
kernel by setting for $x \in \Rd \setminus \seto$

\begin{align}\label{compensated_kernel}
W_{x_0}(x):=\int_0^{\infty} \( p_t(x)-p_t(x_0) \)  \d t,
\end{align}

where $x_0 \in \Rd \setminus \seto$ is an arbitrary but fixed point. From the 
proof of \autoref{lem:compensation_lemma} we immediately obtain the following 
corollary.

\begin{cor}
	$W$ is locally integrable in $\Rd$.
\end{cor}

\begin{thm}\label{thm:recurrent_sweeping_formula}
	Let $x_0 \in D^c$, $d \leq 2$ and $D$ be bounded. Then for $x,y \in D$
	\begin{align}
		G_D(x,y)=W_{x_0}(y-x)-\E^x W_{x_0}(y-X_{\tauD}).
	\end{align}
\end{thm}
	
\begin{proof}
Let $x,y \in D$. Fix $x_0 \in D^c$ and observe that
\begin{align}\label{compensation_proof}
	G_D^{\lambda}(x,y)&=U^{\lambda}(y-x)-\E^x \left[ e^{-\lambda \tauD} 
U^{\lambda}(y-X_{\tauD}) \right] \nonumber \\ 
&=U^{\lambda}(x-y)-U^{\lambda}(x_0)-\E^x \left[ e^{-\lambda \tauD} \left( 
U^{\lambda}(y-X_{\tauD})-U^{\lambda}(x_0)\right) \right] \nonumber \\ &+ 
U^{\lambda}(x_0) \E^x \left[ 1-e^{-\lambda \tauD} \right].
\end{align}
We want to pass with $\lambda$ to $0$. The limit of left-hand side is well 
defined and is equal to $G_D(x,y)$. From \autoref{lem:lambdaUlambda_lem} we get

\begin{align*}
U^{\lambda}(x_0) \E^x \left[ 1-e^{-\lambda \tauD} \right] \leq \lambda 
U^{\lambda}(x_0) \sup_{x \in \Rd} \E^x \tauD \xrightarrow{\lambda \to 0} 0.
\end{align*}

Moreover, from \autoref{lem:compensation_lemma} we obtain that
\begin{align}
\lim_{\lambda \to 0} \(U^{\lambda}(y-x)- U^{\lambda}(x_0) \) = 
W_{x_0}(y-x).
\end{align}

It remains to show the convergence of the middle term of 
\eqref{compensation_proof}. Since $U^{\lambda}$ is radially decreasing, 
$U^{\lambda}(y-X_{\tauD})-U^{\lambda}(x_0)$ is positive on the set $\{y \in 
\Rd\! : \, |y-X_{\tauD}|\leq|x_0|\}$ and non-positive on its complement. By  
\autoref{lem:compensation_lemma}  and the Monotone Convergence Theorem

\begin{align*}
&\lim_{\lambda \to 0} \E^x \left[ e^{-\lambda \tauD} \left( 
U^{\lambda}(y-X_{\tauD})-U^{\lambda}(x_0)\right);|y-X_{\tauD}| < |x_0|\right] \\ 
= &\E^x \left[ W_{x_0}(y-X_{\tauD});|y-X_{\tauD}| < |x_0|\right] \leq 
W_{x_0}(\delta_D(y)) < \infty\,.
\end{align*}

Observe that the left-hand side of \eqref{compensation_proof} converges to $G_D$ 
so it is finite. The remaining integral on the right-hand side converges as well 
by the monotone convergence theorem, but since all the other terms are finite, 
it follows that the integral is also finite and we obtain 
\begin{align*}
	\lim_{\lambda \to 0}  \E^x \left[ e^{-\lambda \tauD} \left( 
U^{\lambda}(y-X_{\tauD})-U^{\lambda}(x_0)\right) \right] = \E^x 
W_{x_0}(y-X_{\tauD})\,,
\end{align*}

which ends the proof.
\end{proof}

\bibliographystyle{abbrv}
\bibliography{bib-reg}
\end{document}